\documentclass{article}

\pdfoutput=1

\usepackage{url}
\usepackage[hmargin=1.5in]{geometry}
\usepackage{enumitem}
\usepackage{amsmath}
\usepackage{amssymb}
\usepackage{amsthm}
\usepackage{graphicx}
\usepackage{environ} 
\usepackage{tikz}

\usetikzlibrary{bbox}
\usetikzlibrary{fadings}
\usetikzlibrary{cd}

\pgfdeclareradialshading{radialedge}{\pgfpointorigin}{%
  color(0bp)=(transparent!0);
  color(20bp)=(transparent!0);
  color(22bp)=(transparent!10);
  color(24bp)=(transparent!90);
  color(25bp)=(transparent!100)
}
\pgfdeclarefading{radial edge}{\pgfuseshading{radialedge}}%

\theoremstyle{plain}
\newtheorem{theorem}{Theorem}[section]
\newtheorem{proposition}[theorem]{Proposition}
\newtheorem{lemma}[theorem]{Lemma}
\newtheorem{corollary}[theorem]{Corollary}
\theoremstyle{definition}
\newtheorem{definition}[theorem]{Definition}
\newtheorem{rmk}[theorem]{Remark}

\setlist[itemize]{topsep=3mm, itemsep=1.8mm, parsep=1.2mm}
\newlist{conditions}{itemize}{1}
\setlist[conditions]{leftmargin=26mm, rightmargin=12mm, topsep=3mm, itemsep=1.8mm, parsep=1.2mm}
\makeatletter
\newcommand{\cond}[1]{\item[(\textsc{#1})]\protected@edef\@currentlabel{\textsc{#1}}}
\newcommand{\condconst}[2]{\item[($\text{\textsc{#1}} \mid #2$)]\protected@edef\@currentlabel{$\text{\textsc{#1}} \mid #2$}}
\makeatother

\newcommand{\maps}{\colon}

\newcommand{\N}{\mathbb{N}}

\newcommand{\R}{\mathbb{R}}
\newcommand{\C}{\mathbb{C}}
\let\Re\relax
\DeclareMathOperator{\Re}{Re}
\newcommand{\laplace}{\mathcal{L}}

\newcommand{\fracderiv}[3]{\partial^{#1}_{#2, #3}}

\newcommand{\cont}{\mathcal{C}}
\newcommand{\holo}{\mathcal{H}}
\newcommand{\singexp}[2]{\mathcal{H}L^\infty_{#1, #2}}
\newcommand{\singexpalg}[1]{\singexp{#1}{\bullet}}

\newcommand{\genvolterra}{\mathcal{W}}
\newcommand{\genker}{w}
\newcommand{\volterra}{\mathcal{V}}
\newcommand{\hardpart}{\mathcal{V}_0}
\newcommand{\softpart}{\mathcal{V}_\star}
\newcommand{\kerwhole}{k}
\newcommand{\hardker}{k_0}
\newcommand{\softker}{k_\star}
\newcommand{\solwhole}{f}
\newcommand{\solproto}{f_0}
\newcommand{\solptb}{f_\star}
\newcommand{\roots}{\mathfrak{B}}

\newcommand{\domain}{\Omega}
\newcommand{\near}{\Omega_\text{near}}
\newcommand{\far}{\Omega_\text{far}}

\newif\ifshowverify
\showverifyfalse
\ifshowverify
\newenvironment{verify}{\color{veriforest}}{\color{black}}
\else
\NewEnviron{verify}{\iffalse\expandafter\BODY\fi}
\fi

\newif\ifhighlightrevs
\highlightrevsfalse
\ifhighlightrevs
\newenvironment{revtwo}{\color{revred}}{\color{black}}
\newcommand{\revtwotext}[1]{\textcolor{revred}{#1}}
\else
\newenvironment{revtwo}{}{}
\newcommand{\revtwotext}[1]{#1}
\fi

\definecolor{ietocean}{RGB}{0, 30, 140}
\definecolor{ietcoast}{RGB}{0, 150, 173}
\definecolor{ietlagoon}{RGB}{0, 216, 180}
\definecolor{veriforest}{RGB}{0, 80, 40}
\definecolor{revred}{RGB}{153, 0, 0}

\pdfpageattr{/Group <</S /Transparency /I true /CS /DeviceRGB>>}

\usepackage[breaklinks]{hyperref}
\hypersetup{
  colorlinks,
  linkcolor={ietcoast},
  citecolor={ietcoast},
  urlcolor={ietcoast}
}

\title{Regular singular Volterra equations on complex domains}
\author{Veronica Fantini and Aaron Fenyes}
\date{}

\begin{document}
\maketitle

\begin{abstract}
The inverse Laplace transform can turn a linear differential equation on a complex domain into an equivalent Volterra integral equation on a real domain. This can make things simpler: for example, a differential equation with irregular singularities can become a Volterra equation with regular singularities. It can also reveal hidden structure, especially when the Volterra equation extends to a complex domain. Our main result is to show that for a certain kind of regular singular Volterra equation on a complex domain, there is always a unique solution of a certain form. As a motivating example, this kind of Volterra equation arises when using Laplace transform methods to solve a {\em level~1} differential equation.
\end{abstract}
\tableofcontents
\section{Introduction}\label{sec:intro}
\subsection{Motivation}\label{sec:motivation}
An ordinary differential equation with an irregular singularity at infinity comes with a special family of formal solutions: ``normal series'' characterized by their formal asymptotics at infinity~\cite{int-irreg}. These series are typically divergent, but they can be ``resummed'' to produce analytic solutions~\cite{loday1994stokes,diverg-resurg--ii,loday-Remy2011,malgrange1995sommation,ramis1991series}. The results of this paper are motivated by the desire to find the same analytic solutions more directly, without the use of formal series. We do this in Section~\ref{sec:example} for a certain class of level~$1$ ODEs, using new Laplace transform methods tailored to the special properties of the analytic solutions that we seek. In~\cite{borel_reg}, we show that these solutions match the ones obtained through resummation. Since the Laplace transform methods we use may be of wider interest, we describe them here in a general and self-contained way.

\begin{revtwo}
In its most basic form, the Laplace transform $\laplace$ turns exponential-type functions of a real ``position'' variable $\zeta$ into holomorphic functions of a complex ``frequency'' variable $z$. It is the integral transform defined by the equation
\[ \laplace \varphi := \int_{\Gamma_\zeta} e^{-z\zeta} \varphi\;d\zeta\,, \]
where $\Gamma_\zeta$ is the ray $\zeta \in (0, \infty)$.     
\end{revtwo}
\begin{verify}[Suppose the magnitude of the Laplace transform integrand is Riemann-integrable on $(0, \infty)$. Write $z$ in terms of its real and imaginary parts as $x + iy$. Since the integrand is holomorphic with respect to $z$, the operator $\frac{\partial}{\partial\overline{z}} = \frac{1}{2}\left(\frac{\partial}{\partial x} + i\frac{\partial}{\partial x}\right)$ annihilates it. Using the comparison version of Leibniz's rule for improper integrals (Theorem~11.10 of A First Course in Real Analysis, by Protter and Morrey) with $\frac{\partial}{\partial x}$ and $\frac{\partial}{\partial y}$, we see that $\frac{\partial}{\partial\overline{z}}$ annihilates the whole integral.]\end{verify} Through identities like
\begin{align*}
\frac{\partial}{\partial z} \laplace \varphi & = \laplace(-\zeta\varphi) \\
\laplace \revtwotext{\psi}\;\laplace\varphi & = \laplace(\revtwotext{\psi} * \varphi) \\
z^{-\nu} \laplace \varphi & = \laplace\,\partial^{-\nu} \varphi\,,
\end{align*}
where $\partial^{-\nu}$ is the Riemann--Liouville fractional integral of order $\nu \in (0, \infty)$, the Laplace transform pulls differential operators on the frequency domain back to Volterra integral operators on the position domain. The favorable regularity properties and comprehensive theory of Volterra equations can thus be brought to bear on linear differential equations.

Some differential equations pull back to Volterra equations with real-analytic kernels, which extend to holomorphic Volterra equations on complex extensions of the position domain. For instance, in Section~\ref{sec:example}, we consider equations of the form
\begin{equation}\label{eqn:intro-level-1}
\left[ P\big(\tfrac{\partial}{\partial z}\big) + \frac{1}{z} Q\big(\tfrac{\partial}{\partial z}\big) + \frac{1}{z^2} R(z^{-1}) \right] \Phi = 0
\end{equation}
given by polynomials $P$, $Q$ and a holomorphic function $R(z^{-1})$ satisfying some extra conditions. A function of the form $\Phi = \laplace \varphi$ satisfies this equation if and only if $\varphi$ satisfies the integral equation
\begin{equation}\label{eqn:intro-use-dict}
\big[ P(-\zeta)+\partial^{-1}\circ Q(-\zeta)+\partial^{-2}\circ R(\partial^{-1}) \big] \varphi = 0\,.
\end{equation}
The integral operator remains well-defined if we make $\zeta$ a complex coordinate and seek solutions $\varphi$ which extend holomorphically into the complex position domain.

Equation~\eqref{eqn:intro-level-1} can be solved by resumming normal series, as described earlier. Each resummed solution is characterized by its $e^{-\alpha z} z^{-\tau_\alpha}$ asymptotics, where $-\alpha$ is a root of $P$ and $\tau_\alpha = Q(-\alpha)/P'(-\alpha)$ is real and positive. To obtain this solution analytically through Laplace transform methods, we must take the Laplace transform $\laplace_{\zeta,\alpha}$ that uses $\zeta = \alpha$ in place of the traditional $\zeta = 0$ as the integration base point. \begin{revtwo}
Explicitly,
\[ \laplace_{\zeta,\alpha} \varphi := \int_{\Gamma_{\zeta,\alpha}} e^{-z\zeta} \varphi\;d\zeta\,, \]
where $\Gamma_{\zeta,\alpha}$ is the ray $\zeta \in \alpha + (0, \infty)$.    
This version of the Laplace transform is defined in Section~\ref{sec:definition_Laplace}. A function of the form $\Phi = \laplace_{\zeta, \alpha} \varphi$ satisfies equation~\eqref{eqn:intro-level-1} if and only if $\varphi$ satisfies the integral equation
\[ \big[ P(-\zeta)+\partial_{\zeta,\alpha}^{-1}\circ Q(-\zeta)+\partial_{\zeta,\alpha}^{-2}\circ R(\partial_{\zeta,\alpha}^{-1}) \big] \varphi = 0\,, \]
where $\partial^{-1}_{\zeta, \alpha}$ is integration with base point $\zeta = \alpha$:
\[ [\partial^{-1}_{\zeta, \alpha} \varphi](a) := \int_{\zeta = \alpha}^a \varphi\;d\zeta\,. \]
We seek a solution $\psi_\alpha$ which
\begin{itemize}
\item has an $O\big((\zeta - \alpha)^{\tau_\alpha-1}\big)$ singularity at $\zeta = \alpha$; and
\item is of exponential type, meaning that it is $O\big(e^{\Lambda|\zeta|}\big)$ for some $\Lambda \in \R$ as $\zeta$ grows.
\end{itemize}
These conditions ensure that $\psi_\alpha$ has a well-defined Laplace transform, and the first condition gives $\laplace_{\zeta, \alpha} \psi_\alpha$ the desired asymptotics.
\end{revtwo}  

The theory of Volterra equations is well-suited to this solution method. By embodying our conditions on $\psi_\alpha$ in a Banach space, defined in Section~\ref{fn-spaces}, we can solve equations like~\eqref{eqn:intro-use-dict} using the contraction mapping theorem, taking advantage of the regularizing effect of the relevant Volterra operators. This establishes the existence and uniqueness of $\psi_\alpha$, as stated in Theorem~\ref{thm:example}. Since integral operators on the position domain correspond directly to multiplication operators on the frequency domain, with no need to worry about initial values, we know immediately that $\laplace_{\zeta, \alpha} \psi_\alpha$ satisfies equation~\eqref{eqn:intro-level-1}, even when $\psi_\alpha$ blows up at $\zeta = \alpha$. By isolating the features of equation~\eqref{eqn:intro-use-dict} that make our contraction mapping argument possible, we get the main result of this paper: the general existence and uniqueness result stated in Section~\ref{sec:results}.

Our results build on known methods for solving differential equations with irregular singularities. Our focus on integral equations in the position domain, and the singularities of their solutions, recalls the approach of Loday-Richaud and Remy, who use a mix of formal and analytic methods inspired by \'{E}calle's theory of resurgence~\cite{loday-Remy2011, EcalleIII}. One of our motivations is to understand resurgence from an analytic perspective, without reference to formal solutions. The Banach spaces we use to achieve this also appear in the work of Braaksma, who uses them to solve systems of ODEs whose coefficients are expressed as Laplace transforms~\cite{braaksma2006laplace}. This overlap suggests an opportunity to combine the two approaches, extending our results to systems of equations, and Braaksma's to equations with more general coefficients.
\subsection{Formalism for coordinates}\label{sec:formalism}
\subsubsection{Rationale}
Our motivating example, equation~\eqref{eqn:intro-use-dict}, has regular singularities at various points $\zeta = \alpha$ on the position domain. Each associated solution $\psi_\alpha$ is most easily expressed in terms of its own translated position coordinate $\zeta_\alpha = \zeta - \alpha$. We therefore expect our results to be used in calculations involving multiple coordinates. Our formalism for points, functions, and coordinates is tailored to such calculations.

For simplicity, we use $\C$ as the position domain throughout this article. In the study of resurgence, however, the position domain is often a Riemann surface over $\C$, or even a more general ``translation surface.'' Thanks to our choice of formalism, our results and their proofs generalize immediately to this setting.
\subsubsection{Points and coordinates}
From now on, we will carefully respect the distinction between points and coordinates. A point is a location in space, while a coordinate is a function on space. For example, the position coordinates $\zeta$ and $\zeta_\alpha$ are functions on the position domain, just like the integral equation solutions $\psi_\alpha$. A coordinate, like $\zeta$, can be evaluated at a point $b$, yielding a number $\beta := \zeta(b)$. Conversely, the point $b$ can be described as the place where $\zeta = \beta$.
\subsubsection{Integration}
When we integrate, the integrand will always be either a 1-form, like $\varphi\,d\zeta$, or a density, like $|\varphi\,d\zeta|$~\cite[Section~1.8]{local-viewpoint}. On a simply connected domain, all integration paths are equivalent, so we will typically specify an integral by its start and end points. These points may be given directly, as in the expression $\int_b^a \varphi\,d\zeta$, or described in terms of coordinates, as in the equivalent expression $\int_{\zeta = \beta}^a \varphi\,d\zeta$. In the integral
\[ \int_{\zeta = \beta}^a \big(\zeta(a) - \zeta\big)^{-1/2} \varphi\,d\zeta\,, \]
notice that $\zeta(a)$ is a number, which stays constant throughout the integration, while $\zeta$ is a function, whose value changes along the integration path.

\begin{revtwo}
Sometimes, in expressions like
\[ \int_{\zeta(a') = \beta}^a w(a, a')\,\varphi(a')\;d\zeta(a')\,, \]
it is useful to give a name $a'$ to the point moving along the integration path. In this case, all of the functions and differential forms whose values change along the integration path are explicitly evaluated at $a'$.
\end{revtwo}
\subsubsection{Plugging into formal expressions}
Plugging a function $\varphi$ into a formal polynomial or power series $P$ gives a new function $P(\varphi)$. The correspondence between operators on the position and frequency domains will lead to expressions of this kind. In Section~\ref{sec:motivation}, for example, we saw that the differential operator $P\big(\tfrac{\partial}{\partial z}\big)$ on the frequency domain corresponds to multiplication by the function $P(-\zeta)$ on the position domain.
\subsection{Notation}
\subsubsection{Uniform bounds}
Given a complex-valued function $\varphi$ and a nonnegative, real-valued function $\omega$, we will write $|\varphi| \lesssim \omega$ to say that $|\varphi|$ is bounded by a constant multiple of $\omega$. Unless we say otherwise, the bound holds throughout the domain where both functions are defined. If the bound involves an explicit variable, we will say that it holds over all values of the variable. For example, we will write ``$|R_j| \lesssim \lambda^j$ over all $j \in \{0, 1, 2, \ldots\}$'' to mean ``there is a constant $M$ such that for all $j \in \{0, 1, 2, \ldots\}$, the inequality $|R_j| \le M \lambda^j$ holds.''
\subsubsection{Boundary points}
We will say that a subset of a topological space {\em touches} the points in its closure. This is a way of expressing the nearness relation associated with the topology~\cite[Chapter~5, Definition~2.11]{joshi1983gen-top}. If $\Omega \subset \C$ touches the point $p \in \C$, we define a neighborhood of $p$ in $\Omega$ to be $U \cap \Omega$ for some open neighborhood $U \subset \C$ of $p$. This agrees with the usual definition of a neighborhood when $p$ is an element of $\Omega$.
\subsection{Setting}\label{setting}
\subsubsection{The domain}\label{setting:domain}
Throughout this paper, as described in Section~\ref{sec:motivation}, the position variable $\zeta$ will be the standard coordinate on $\C$. 

Take a simply connected open set $\domain \subset \C$ that touches but does not contain $\zeta = 0$ and that satisfies the following condition.
\begin{conditions}
\cond{star}\label{cond:star} The set $\domain$ is star-shaped around $\zeta = 0$. In other words, for any $a \in \domain$, a straight path from $\zeta = 0$ to $a$ stays in $\domain$. Since $\domain$ does not contain $\zeta = 0$, we will always omit that starting point from the path.
\end{conditions}
For the applications we have in mind, $\domain$ typically resembles the set pictured below.
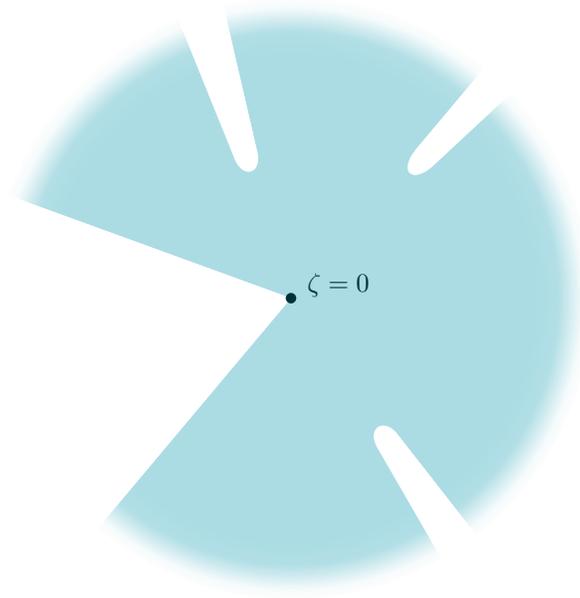
\vspace{2mm}

\begin{center}
\begin{tikzpicture}[scale=0.9]
\newcommand{\spill}{4}
\fill[ietcoast!33, bezier bounding box, path fading=radial edge]
  (-\spill, -\spill) (\spill, \spill)
  (0, 0) -- (160:\spill)
  arc (160:112:\spill) -- (112:2) .. controls (112:1.7) and (103:1.7) .. (103:2) -- (103:\spill)
  arc (103:50:\spill) -- (50:2.6) .. controls (50:2.2) and (43:2.2) .. (43:2.6) -- (43:\spill)
  arc (43:-52:\spill) -- (-52:2.3) .. controls (-52:2) and (-60:2) .. (-60:2.3) -- (-60:\spill)
  arc (-60:-130:\spill) -- (0, 0);
\fill[ietcoast!33!black] circle (0.7mm) node[anchor=195, outer sep=1mm] {$\zeta = 0$};
\end{tikzpicture}
\end{center}
\subsubsection{The prototype operator}\label{setting:basic}
The prototypical example of the kind of operator we will be working with is a holomorphic Volterra operator $\hardpart$ with a separable kernel $\hardker$ and a regular singularity at $\zeta = 0$.

Let $\holo(\domain)$ be the space of holomorphic functions on $\domain$.
\begin{definition}\label{defn:volterra}
A {\em holomorphic Volterra operator} on $\domain$ is a linear operator
\[\genvolterra\colon\holo(\domain)\to\holo(\domain) \]
defined by an integral
\begin{revtwo}
\[ [\genvolterra\,\varphi](a) = \int_{\zeta(a') = 0}^a \genker(a, a')\,\varphi(a')\;d\zeta(a') \]
\end{revtwo}
in which the {\em kernel} $w$ is a holomorphic function on $\Omega^2$.
\end{definition}
\begin{definition}
A holomorphic Volterra operator is {\em separable} if its kernel factors as a product of two functions, one on each factor of $\Omega^2$.
\end{definition}
For the prototype operator $\hardpart$, we will suppose that the kernel $\hardker$ can be written as a ratio
\[ \hardker(a, a') = - \frac{q(a')}{p(a)}, \]
for some $p, q \in \mathcal{H}(\Omega)$. \begin{revtwo}Explicitly,
\[ \big[\hardpart \varphi\big](a) = - \int_{\zeta(a')=0}^{a} \frac{q(a')}{p(a)}\,\varphi(a')\;d\zeta(a')\,. \]
\end{revtwo}

For the purposes of this paper, $\hardpart$ has a {\em regular singularity} at $\zeta=0$ if the following condition holds.
\begin{conditions}
\condconst{sing}{\tau}\label{cond:sing}For some constant $\tau > 0$, the difference
\[ \hardker(a, a) - \frac{\tau}{\zeta(a)} \]
is bounded on a neighborhood of $\zeta(a) = 0$ in $\domain$.
In addition, for each $\sigma > \tau$, there is a neighborhood of $\big(\zeta(a), \zeta(a')\big) = (0, 0)$ in $\Omega^2$ on which
\[ |\hardker(a, a')| < \frac{\sigma}{|\zeta(a)|}. \]
\end{conditions}
For most of our results, we will need to make sure that $\domain$ does not touch any singularities of $\hardker$ other than the one at $\zeta = 0$. We will also need to control $\hardker(a, a')$ when $a$ is away from $\zeta = 0$, requiring $\hardker$ to be bounded on the diagonal in $\domain^2$ and to grow at most exponentially as we move away from the diagonal. These requirements can be combined into one condition.
\begin{conditions}
\condconst{diag$_0$}{\lambda_\Delta}\label{cond:diag-basic} For some constant $\lambda_\Delta$, we have
\[ |\hardker(a, a')| \lesssim \frac{1}{|\zeta(a)|} e^{\lambda_\Delta|\zeta(a)-\zeta(a')|} \]
over all $a, a' \in \domain$.
\end{conditions}
This condition explains the shape of the example domain illustrated in Section~\ref{setting:domain}. As $\domain$ stretches out toward infinity, it has to part around the zeros of $p$, keeping well away from every zero except the one at $\zeta = 0$.
\begin{verify}
By Condition \eqref{cond:diag-basic} we can find $c>0$ such that \[ | \hardker(a, a') | \le \frac{c}{|\zeta(a)|} e^{\lambda_\Delta|\zeta(a)-\zeta(a')|} \]
then setting $\delta=\frac{\log(\sigma/c)}{2\lambda_\Delta}$, we deduce that for every $|\zeta(a)|\le\delta$ and  $|\zeta(a')|\le\delta$
\begin{align*}
|\hardker(a,a')|&\le \frac{c}{|\zeta(a)|} e^{\lambda_\Delta|\zeta(a)-\zeta(a')|}\\
&\le \frac{c}{|\zeta(a)|} e^{\lambda_\Delta(|\zeta(a)|+|\zeta(a')|)}\\
&\frac{c}{|\zeta(a)|} e^{2\lambda_\Delta \delta}\\
&=\frac{\sigma}{|\zeta(a)|}
\end{align*}
\end{verify}

We will occasionally mention an optional condition on $p$ that allows us to state our main results more explicitly.
\begin{conditions}
\condconst{reg-p}{B, \epsilon}\label{cond:reg-p}
For some nonzero constant $B$ and some $\epsilon > 0$,
\[ p \in B\zeta + O\big(|\zeta|^{1 + \epsilon}\big) \]
at $\zeta = 0$.
\end{conditions}
\subsubsection{The perturbed operator}\label{setting:perturbed}

We now perturb $\hardpart$ to a more general operator $\volterra = \hardpart +\softpart$. \begin{revtwo}The perturbation
\[ [\softpart\,\varphi](a) := \int_{\zeta(a') = 0}^a \softker(a, a')\,\varphi(a')\;d\zeta(a') \]
will be nonseparable,\end{revtwo}\footnote{Unless $\softpart$ is zero, of course.} but it will have a regularizing effect, as we will show in Proposition \ref{prop:smoothing}. To get the smoothing effect, we will require \revtwotext{the kernel $\softker$} to vanish to some order $\gamma > 0$ on the diagonal in $\domain^2$. This requirement, combined with two others, will be made precise in Condition~\eqref{cond:eps-lambda}.

Since $\softpart$ is a holomorphic Volterra operator, $\softker$ is a holomorphic function on $\domain^2$. We will allow $\softker(a, a')$ to have a simple pole at $\zeta(a) = 0$, like $\hardker(a, a')$ does, but we will not allow any sharper singularity. We will also put an exponential bound on the growth of $\softker$ away from the diagonal, mimicking Condition~\eqref{cond:diag-basic} on $\hardker$. Altogether, we will require the following:
\begin{conditions}
\condconst{diag$_\star$}{\gamma, \lambda_\Delta}\label{cond:eps-lambda} There is a constant $\gamma > 0$ for which
\[ |\softker(a, a')| \lesssim\frac{|\zeta(a)-\zeta(a')|^\gamma}{|\zeta(a)|}\,e^{\lambda_\Delta|\zeta(a)-\zeta(a')|}\]
over all $a, a' \in \domain$.
\end{conditions}
Notice that this condition prevents $\softker$ from being separable---unless it is zero, of course.

Like $\hardker$, the combined kernel $k = \hardker + \softker$ of $\volterra$ grows in a controlled way when its arguments are near $\zeta = 0$, and when the difference between its arguments grows. We will provide specific bounds in Section~\ref{sec:bounds on k}. 

\subsection{Main results}\label{sec:results}
We want to show that when $\volterra$ is a regular singular Volterra operator of the kind described in Section~\ref{setting:perturbed}, the equation $f = \volterra f$ has a unique solution of a certain form. For the prototypical operator $\hardpart$ described in Section~\ref{setting:basic}, this solution can be written explicitly.
\begin{theorem}\label{thm:basic_volterra}
\begin{revtwo}Suppose the operator
\[ \big[\hardpart \varphi\big](a) := - \int_{\zeta(a')=0}^{a} \frac{q(a')}{p(a)}\,\varphi(a')\;d\zeta(a') \]
\end{revtwo}
satisfies {\em Condition \eqref{cond:sing}}. Then the equation
\begin{equation}\label{eq:hardpart}
f = \hardpart f    
\end{equation}
has the {\em prototype solution}
\begin{equation}\label{eqn:test_solution}
\solproto(a) = \frac{1}{p(a)} \exp\left(-\int_{b}^{a}\frac{q}{p}\;d\zeta\right).
\end{equation}
Changing the base point $b \in \domain$ only multiplies $f_0$ by a nonzero constant.
\end{theorem}
We will prove this result in Section~\ref{sec:construction}.

The solution $\solproto$ from Theorem~\ref{thm:basic_volterra} has, at worst, a mild power-law singularity at $\zeta = 0$. When $\hardpart$ is especially regular, $\solproto$ also grows at most exponentially as $|\zeta| \to \infty$. The function spaces defined in Section~\ref{fn-spaces} express both of these regularity properties. \begin{revtwo}We will state our results in terms of the spaces $\singexpalg{\sigma}(\domain)$, which are defined as unions of normed spaces. Every function $f \in \singexpalg{\sigma}(\domain)$ has a finite norm
\[ \|f\|_{\sigma,\Lambda} := \sup_\Omega |\zeta|^{-\sigma} e^{-\Lambda|\zeta|}\,|f| \]
for some $\Lambda \in \R$.
\end{revtwo}
\begin{theorem}\label{thm:proto-growth}
\begin{revtwo}Suppose the operator
\[ \big[\hardpart \varphi\big](a) := - \int_{\zeta(a')=0}^{a} \frac{q(a')}{p(a)}\,\varphi(a')\;d\zeta(a') \]
\end{revtwo}
satisfies {\em Condition~\eqref{cond:sing}}. Then, on a small enough neighborhood of $\zeta = 0$, we have $|\solproto| \lesssim |\zeta|^{\tau-1}$.

Suppose $\hardpart$ also satisfies {\em Condition~\eqref{cond:diag-basic}}. Then $f_0$ belongs to the space $\singexpalg{\tau-1}(\domain)$ defined in Section~\ref{fn-spaces}.
\end{theorem}
We will prove this result in Section~\ref{sec:asymptotics}.
\begin{rmk}
When $\hardpart$ also satisfies Condition~\eqref{cond:reg-p}, we can get a better estimate of the prototype solution near $\zeta = 0$, as described in Proposition~\ref{prop:better-proto-estimate} from Section~\ref{sec:asymptotics}.
\end{rmk}

The perturbed equation $f = \volterra f$ is more complicated, but its regular singularity at $\zeta = 0$ is essentially the same. We might therefore expect a solution that looks like $\solproto$ near the singularity, differing only by a less singular perturbation. If we strengthen the constraints on $\hardpart$ a little more, this expectation is fulfilled.
\begin{theorem}\label{thm:general_volterra}
Suppose $\hardpart$ satisfies {\em Conditions \eqref{cond:sing}} and \eqref{cond:diag-basic}, and \begin{revtwo}the perturbation
\[ [\softpart\,\varphi](a) := \int_{\zeta(a') = 0}^a \softker(a, a')\,\varphi(a')\;d\zeta(a') \]
\end{revtwo}
satisfies {\em Condition~\eqref{cond:eps-lambda}}. Then the equation
\[ f = \volterra f, \]
\revtwotext{where $\volterra = \hardpart + \softpart$,} has a unique solution $f$ in the affine subspace
\[ f_0 + \singexpalg{\tau-1+\gamma}(\Omega) \]
of the space $\singexpalg{\tau-1}(\Omega)$ defined in Section~\ref{fn-spaces}. Here, $f_0$ is the prototype solution \eqref{eqn:test_solution} from Theorems \ref{thm:basic_volterra}--\ref{thm:proto-growth}, which belongs to the space $\singexpalg{\tau-1}(\domain)$.
\end{theorem}
\begin{corollary}\label{cor:expand_uniq}
Under the conditions of Theorem~\ref{thm:general_volterra}, for any $\rho > \tau$, the equation $f = \volterra f$ has at most one solution in $f_0 + \singexpalg{\rho-1}(\Omega)$. 
\end{corollary}
This result will follow from a more general result about inhomogeneous equations.
\begin{lemma}\label{lem:perturbed_volterra}
Suppose $\hardpart$ satisfies {\em Conditions \eqref{cond:sing}} and \eqref{cond:diag-basic}, and $\softpart$ satisfies {\em Condition \eqref{cond:eps-lambda}}. Suppose we are also given a function $g$, which for some $\rho > \tau$ belongs to the space $\singexpalg{\rho-1}(\Omega)$ defined in Section~\ref{fn-spaces}. Then the inhomogeneous equation
\[ f = \volterra f + g, \]
has a unique solution $f$ in the space $\singexpalg{\rho-1}(\Omega)$.
\end{lemma}
We will prove Lemma~\ref{lem:perturbed_volterra} in Sections \ref{sec:V is a contraction}--\ref{sec:existence and uniqueness}, using the contraction mapping theorem. The heart of the argument is Proposition~\ref{prop:get-contraction}, which shows us how to find relevant subspaces where $\volterra$ is a contraction.

We will reduce Theorem~\ref{thm:general_volterra} to Lemma~\ref{lem:perturbed_volterra} in Section~\ref{sec:existence and uniqueness}, by rewriting the homogeneous equation $f = \volterra f$ in the subspace $f_0 + \singexpalg{\tau-1+\gamma}(\Omega)$ as an inhomogeneous equation in a more regular space. To show that the inhomogeneous term $\softpart f_0$ is regular enough, we will use Proposition~\ref{prop:smoothing} from Section~\ref{sec:image under soft_part} to prove that $\softpart$ improves on the regularity ascribed by Theorem~\ref{thm:proto-growth} to the prototype solution $f_0$.
\begin{rmk}
When $\hardpart$ also satisfies Condition~\eqref{cond:reg-p}, Theorem~\ref{thm:general_volterra} can be restated to give a unique solution in the affine subspace
\[ \zeta^{\tau-1} + \singexpalg{\rho-1}(\domain) \]
when $\rho > \tau$ is low enough, as described in Proposition~\ref{prop:alt-general_volterra} from Section~\ref{sec:existence and uniqueness}.
\end{rmk}
\subsection{Acknowledgements}
We are grateful for the conversations with Kihyun Kim, Angeliki Menegaki, and Sung-Jin Oh that informed the early stages of this work. This paper is a result of the ERC-SyG project, Recursive and Exact New Quantum Theory (ReNewQuantum) which received funding from the European Research Council (ERC) under the European Union's Horizon 2020 research and innovation programme under grant agreement No 810573. It was also made possible by the hospitality of the Institut des Hautes \'{E}tudes Scientifiques during a return visit by AF, funded by the Fondation Math\'{e}matique Jacques Hadamard under the \textit{Junior Scientific Visibility} program.
\section{Function spaces for holomorphic Volterra operators}\label{fn-spaces}
\subsection{Weighted holomorphic $L^{\infty}$ spaces}\label{sec:fn-space-defs}
Throughout this paper, as described in Section~\ref{sec:motivation}, the position variable $\zeta$ will be the standard coordinate on $\C$. Take a simply connected open set $\domain \subset \C$ that touches but does not contain $\zeta = 0$. Let $\cont(\domain)$ be the space of continuous complex-valued functions on $\domain$, equipped with the compact-open topology. This is the coarsest topology in which the seminorm $f \mapsto \sup_K |f|$ is continuous for every compact subset $K \subset \domain$~\cite[Example~2.6 and Section~4 notes]{fnl-cpx-anal}.

The holomorphic functions form a closed subspace $\holo(\domain) \subset \cont(\domain)$~\cite[Proposition~3.14]{fnl-cpx-anal}. We first restrict our attention to holomorphic functions on $\domain$ which are uniformly of exponential type $\Lambda \in \R$.
\begin{definition}\label{def:unif-exp}
For any $\Lambda \in \R$, let
\[ \singexp{0}{\Lambda}(\domain) \subset \holo(\domain) \]
be the subspace consisting of functions $f$ with $|f| \lesssim e^{\Lambda|\zeta|}$ over $\domain$, equipped with the norm
\[ \|f\|_{0, \Lambda} := \sup_\Omega e^{-\Lambda|\zeta|}\,|f|. \]
\end{definition}
With respect to the seminorm on $\holo(\domain)$ given by a compact set $K \subset \domain$, the inclusion map $\singexp{0}{\Lambda}(\domain) \hookrightarrow \holo(\domain)$ has norm $\sup_K e^{\Lambda |\zeta|}$. That means the inclusion is continuous.
\begin{rmk}
The subspace of functions of exponential type $\Lambda$ strictly contains $\singexp{0}{\Lambda}(\domain)$. A function $f$ is of exponential type $\Lambda$ if for every $\varepsilon>0$, there is a constant $A_\varepsilon$ such that $|f|\le A_\varepsilon e^{(\Lambda+\varepsilon)|\zeta|}$. Definition~\ref{def:unif-exp} instead requires a uniform constant $A$ such that $|f| \le A e^{\Lambda|\zeta|}$.\end{rmk}
We now relax our norm to allow both exponential growth at infinity and a power-law singularity at $\zeta = 0$.
\begin{definition}
For any $\sigma, \Lambda \in \R$, let
\[ \singexp{\sigma}{\Lambda}(\domain) \subset \holo(\domain) \]
be the subspace consisting of functions $f$ with $|f| \lesssim |\zeta|^\sigma e^{\Lambda|\zeta|}$ over $\domain$, equipped with the norm
\[ \|f\|_{\sigma,\Lambda} := \sup_\Omega |\zeta|^{-\sigma} e^{-\Lambda|\zeta|}\,|f|. \]
\end{definition}
The inclusion $\singexp{\sigma}{\Lambda}(\domain) \hookrightarrow \holo(\domain)$ is again continuous.
\begin{proposition}\label{exp-complete}
The space $\singexp{\sigma}{\Lambda}(\domain)$ is complete.
\end{proposition}
\begin{proof}
Take a Cauchy sequence $(f_j)_{j \in \N}$ in $\singexp{\sigma}{\Lambda}(\domain)$. This sequence is Cauchy in $\holo(\domain)$ as well, because the inclusion map $\singexp{\sigma}{\Lambda}(\domain) \hookrightarrow \holo(\domain)$ is bounded with respect to each of the seminorms on $\holo(\domain)$ given by $f \mapsto \sup_K |f|$ for compact subsets $K \subset \domain$. Since $\holo(\domain)$ is complete~\cite[Proposition~3.5]{fnl-cpx-anal},\footnote{That is, a sequence which is Cauchy in each of the seminorms on $\holo(\domain)$ will always converge in the topology of $\holo(\domain)$, which is the coarsest topology in which all of the seminorms are continuous.} our sequence converges to a function $f$ there.

The Cauchy property in $\singexp{\sigma}{\Lambda}(\domain)$ says that for any $r > 0$, there is some $n$ for which $|\zeta|^{-\sigma} e^{-\Lambda |\zeta|}\,|f_k - f_n| \le r$ whenever $k \ge n$. Since convergence in $\holo(\domain)$ implies pointwise convergence, it follows that $|\zeta|^{-\sigma} e^{-\Lambda |\zeta|}\,|f - f_n| \le r$. This reasoning shows that $(f_j)_{j \in \N}$ converges to $f$ in the norm $\|\cdot\|_{\sigma, \Lambda}$. It also shows that $f$ is in $\singexp{\sigma}{\Lambda}(\domain)$: picking some $r > 0$, we have
\begin{align*}
|\zeta|^{-\sigma} e^{-\Lambda |\zeta|}\,|f| & \le |\zeta|^\sigma e^{-\Lambda |\zeta|}\,|f - f_n| + |\zeta|^{-\sigma} e^{-\Lambda |\zeta|}\,|f_n| \\
& \le r + \|f_n\|_{\sigma,\Lambda}
\end{align*}
for the corresponding $n$, so $|\zeta|^{-\sigma} e^{-\Lambda |\zeta|}\,|f|$ is bounded.
\end{proof}
\subsection{Continuous inclusions between different $\singexp{\sigma}{\Lambda}(\Omega)$}\label{sec:inclusions}
\begin{proposition}\label{prop:inclus-ge-exp}
If $\Lambda'\leq\Lambda$, the inclusion map $\singexp{\sigma}{\Lambda'}(\Omega)\hookrightarrow \singexp{\sigma}{\Lambda}(\Omega)$ is continuous.
\end{proposition}
\begin{proof}
By definition,
\[ \|f\|_{\sigma,\Lambda}=\sup_{\Omega} |\zeta|^{-\sigma}\,e^{-\Lambda |\zeta|}\, |f|. \]
The norm $\|\cdot\|_{\sigma, \Lambda'}$ is designed to give $|f| \le |\zeta|^\sigma\,e^{\Lambda'|\zeta|}\,\|f\|_{\sigma, \Lambda'}$, implying that
\begin{align*}
\|f\|_{\sigma,\Lambda} & \leq \sup_{\Omega} |\zeta|^{-\sigma}\,e^{-\Lambda |\zeta|}\,|\zeta|^\sigma\,e^{\Lambda'|\zeta|}\,\|f\|_{\sigma, \Lambda'}\\
&=\sup_{\Omega} e^{-(\Lambda-\Lambda') |\zeta|}\,\|f\|_{\sigma, \Lambda'}\\
&\leq \|f\|_{\sigma,\Lambda'}.
\end{align*}
In the last step, we use the assumption that $\Lambda' \le \Lambda$.
\end{proof}
The union $\bigcup_{\Lambda \in \R} \singexp{0}{\Lambda}(\domain)$ is the subspace comprising all functions of exponential type. The continuous inclusions between the $\singexp{0}{\Lambda}(\domain)$ define a limit topology on this subspace. The spaces $\singexp{\sigma}{\Lambda}(\domain)$ fit together analogously for other $\sigma \in \R$.
\begin{definition}\label{def:exp-top}
For any $\sigma \in \R$, let $\singexpalg{\sigma}(\domain)$ be the union
\[ \bigcup_{\Lambda \in \R} \singexp{\sigma}{\Lambda}(\domain) \]
equipped with the finest topology that makes all the inclusions $\singexp{\sigma}{\Lambda}(\domain) \hookrightarrow \singexpalg{\sigma}(\domain)$ continuous.
\end{definition}
\begin{proposition}\label{prop:inclus-lt-pow-gt-exp}
If $\sigma'>\sigma$ and $\Lambda'<\Lambda$, the inclusion map $\singexp{\sigma'}{\Lambda'}(\Omega)\hookrightarrow \singexp{\sigma}{\Lambda}(\Omega)$ is continuous.
\end{proposition}
\begin{proof}
By definition,
\begin{align*}
\|f\|_{\sigma,\Lambda}&=\sup_{\Omega} |\zeta|^{-\sigma}  e^{-\Lambda |\zeta|} |f|\\
&= \sup_{\Omega} |\zeta|^{\sigma'-\sigma}\,|\zeta|^{-\sigma'}\,e^{-\Lambda'|\zeta|}\,  e^{-(\Lambda-\Lambda') |\zeta|} \, |f|.
\end{align*}
The function $|\zeta|^{\sigma'-\sigma}\,  e^{-(\Lambda-\Lambda') |\zeta|}$ is bounded near $\zeta = 0$ because the power of $|\zeta|$ is positive, and it is bounded far from $\zeta = 0$ thanks to the decaying exponential. Hence,
\begin{align*}
\|f\|_{\sigma,\Lambda}&\leq C\sup_\Omega  |\zeta|^{-\sigma'}\, e^{-\Lambda'|\zeta|} \, |f|\\
&=C \|f\|_{\sigma',\Lambda'}
\end{align*}
for $C = \sup_{\Omega}  |\zeta|^{\sigma'-\sigma}\,  e^{-(\Lambda-\Lambda') |\zeta|}$.
\end{proof}
\begin{proposition}\label{prop:inclus-lt-pow-alg}
When $\sigma' > \sigma$, there is a continuous inclusion $\singexpalg{\sigma'}(\domain)\hookrightarrow \singexpalg{\sigma}(\domain)$.
\end{proposition}
\begin{proof}
For each $\Lambda'$, we define a continuous inclusion $\singexp{\sigma'}{\Lambda'}(\domain) \hookrightarrow \singexpalg{\sigma}(\domain)$ by choosing some $\Lambda > \Lambda'$ and composing the continuous inclusions
\begin{center}
\begin{tikzcd}[column sep=25mm, labels={inner sep=0.75ex}]
\singexp{\sigma'}{\Lambda'}(\domain) \arrow[r, hook, "\text{\footnotesize Proposition~\ref{prop:inclus-lt-pow-gt-exp}}"'] & \singexp{\sigma}{\Lambda}(\domain) \arrow[r, hook, "\text{\footnotesize Definition~\ref{def:exp-top}}"'] & \singexpalg{\sigma}(\domain).
\end{tikzcd}
\end{center}
For any $\Lambda'' < \Lambda'$, the inclusions
\begin{center}
\begin{tikzcd}[row sep=6mm, column sep=8mm, labels={inner sep=0.75ex}]
\singexp{\sigma'}{\Lambda''}(\domain) \arrow[rr, hook, "\text{\footnotesize Proposition~\ref{prop:inclus-ge-exp}}"] \arrow[rd, hook] & & \singexp{\sigma'}{\Lambda'}(\domain) \arrow[ld, hook] \\
& \singexpalg{\sigma}(\domain)
\end{tikzcd}
\end{center}
automatically commute, because as nontopological linear maps they are all inclusions between subspaces of $\holo(\domain)$. The existence of the desired continuous inclusion $\singexpalg{\sigma'}(\Omega)\hookrightarrow \singexpalg{\sigma}(\Omega)$ then follows from the universal property of the limit topology.
\end{proof}
\section{Solving holomorphic Volterra equations}\label{sec:proof_main_results}
\subsection{Overview}
In this section, we prove the results stated in Section~\ref{sec:results}, which lay out a method for solving the regular singular Volterra equation $\solwhole = \volterra \solwhole$.

We first construct the prototype solution $\solproto$ from the kernel of $\hardpart$. We show in Section~\ref{sec:proto-construction-regularity} that $\solproto$ satisfies the equation $\solproto = \hardpart \solproto$ and belongs to the space $\singexp{\tau-1}{\lambda_0}(\domain)$.

We then study how the perturbation $\softpart$ affects $\solproto$. We show in Section~\ref{sec:image under soft_part} that $\softpart$ has a smoothing effect, reducing the sharpness of any power-law singularity at $\zeta = 0$. In particular, it sends $\solproto$ into $\singexpalg{\tau-1+\gamma}(\domain)$.

We conclude that $\volterra$ sends $\solproto$ into the affine subspace $\solproto + \singexpalg{\tau-1+\gamma}(\domain)$, suggesting that the equation $\solwhole = \volterra \solwhole$ has a solution there. To confirm this, we show in Section~\ref{sec:V is a contraction} that $\volterra$ is a contraction of $\singexp{\tau-1+\gamma}{\Lambda}(\domain)$ when $\Lambda$ is large enough. This implies that $\volterra$ has a unique fixed point in $\solproto + \singexpalg{\tau-1+\gamma}(\domain)$, as we show in Section~\ref{sec:existence and uniqueness}.
\subsection{Construction and regularity of the prototype solution \\ \textit{(proof of Theorems~\ref{thm:basic_volterra}--\ref{thm:proto-growth})}}\label{sec:proto-construction-regularity}
\subsubsection{Construction}\label{sec:construction}
\begin{proof}[Proof of Theorem~\ref{thm:basic_volterra}]
Rewrite $\solproto$ as $(1/p)\,\chi$, where
\[ \chi(a) = \exp\left(-\int_{b}^{a}\frac{q}{p}\;d\zeta\right). \]
Observing that $d\chi = -(q/p)\,\chi\;d\zeta$ simplifies the calculation of $\hardpart \solproto$. For each $a \in \domain$,
\begin{align*}
\big[\hardpart \solproto\big](a) &= - \int_{\zeta=0}^{a} \frac{q}{p(a)}\,\solproto\;d\zeta \\
& = -\int_{\zeta=0}^{a} \frac{q}{p(a)}\,\frac{1}{p}\,\chi\;d\zeta \\
& = - \frac{1}{p(a)}  \int_{\zeta=0}^{a} \frac{q}{p}\,\chi\;d\zeta\\
& = \frac{1}{p(a)} \int_{\zeta=0}^{a}\;d\chi \\
& = \frac{1}{p(a)} \left[ \chi(a) - \lim_{\zeta \to 0} \chi \right] \\
& = \solproto(a) - \frac{1}{p(a)} \lim_{\zeta \to 0} \chi.
\end{align*}
To prove that $\hardpart \solproto = \solproto$, we must now show that $\lim_{\zeta \to 0} \chi = 0$.

By Condition~\eqref{cond:sing}, there exist a radius $\delta>0$ and a constant $C$ such that
\begin{equation}\label{eqn:sing-bound}
\left|\frac{q}{p} + \frac{\tau}{\zeta}\right| < C
\end{equation}
whenever $|\zeta| < \delta$. To use this bound, we rewrite the integral in the definition of $\chi$:
\begin{align*}
-\int_b^a \frac{q}{p}\;d\zeta & = \int_b^a \frac{\tau}{\zeta}\;d\zeta - \int_b^a \left( \frac{q}{p} + \frac{\tau}{\zeta} \right)\;d\zeta \\
& = \tau \log\left(\frac{\zeta(a)}{\zeta(b)}\right) - \int_b^a \left( \frac{q}{p} + \frac{\tau}{\zeta} \right)\;d\zeta.
\end{align*}
Exponentiating both sides shows that
\[ \chi(a) = \left(\frac{\zeta(a)}{\zeta(b)}\right)^\tau \exp\left[-\int_b^a \left( \frac{q}{p} + \frac{\tau}{\zeta} \right)\;d\zeta\right]. \]
Recalling that a change of base point just multiplies $\solproto$ by a nonzero constant, choose the base point $b \in \Omega$ so that $|\zeta(b)| = \delta/2$. The exponential factor in the formula above is then bounded between $\exp\big({-\tfrac{3}{2}C\delta}\big)$ and $\exp\big(\tfrac{3}{2}C\delta\big)$ whenever $|\zeta(a)| < \delta$. Since $\tau$ is positive, this is enough to show that $\lim_{\zeta \to 0} \chi = 0$. It follows, as discussed above, that $\hardpart \solproto = \solproto$.
\end{proof}
\subsubsection{Regularity}\label{sec:asymptotics}
Theorem~\ref{thm:proto-growth} comprises two results with different conditions. We will prove them separately as Propositions \ref{prop:asymptotic at zero} and \ref{prop:asymptotic at infinity}.
\begin{proposition}\label{prop:asymptotic at zero}
Suppose $\hardpart$ satisfies {\em Condition~\eqref{cond:sing}}. Then, on a small enough neighborhood of $\zeta = 0$, we have $|\solproto| \lesssim |\zeta|^{\tau-1}$.
\end{proposition}

\begin{proof}
Go back to the proof of Theorem~\ref{thm:basic_volterra}, where we found a radius $\delta > 0$ and a constant $C$ such that inequality~\eqref{eqn:sing-bound} holds whenever $|\zeta| < \delta$. The subsequent argument, which we used to show that $\lim_{\zeta \to 0} \chi = 0$, actually supports a stronger conclusion: it shows that $|\chi| \lesssim |\zeta|^\tau$ on the region $|\zeta| < \delta$. This implies that
\[ |\solproto| \lesssim \left|\frac{1}{p}\right|\,|\zeta|^\tau \]
on the region $|\zeta| < \delta$.

Choose some $\sigma > \tau$. By Condition~\eqref{cond:sing}, there is a radius $r < \delta$ for which
\[ \left|\frac{q(a')}{p(a)}\right| < \frac{\sigma}{|\zeta(a)|} \]
whenever $|\zeta(a)| < r$ and $|\zeta(a')| < r$. Choosing a point $b'$ with $|\zeta(b')| < r$ and $q(b') \neq 0$, we deduce that
\begin{align*}
|\solproto| & \lesssim \left|\frac{1}{q(b')}\right|\,\left|\frac{q(b')}{p}\right|\,|\zeta|^\tau \\
& \lesssim \left|\frac{1}{q(b')}\right|\,\frac{\sigma}{|\zeta|}\,|\zeta|^\tau \\
& \lesssim |\zeta|^{\tau-1}
\end{align*}
on the region $|\zeta| < r$, as desired.
\end{proof}
\begin{proposition}\label{prop:asymptotic at infinity}
Suppose $\hardpart$ satisfies {\em Conditions \eqref{cond:sing}} and \eqref{cond:diag-basic}. Then $\solproto$ belongs to the space $\singexpalg{\tau-1}(\domain)$.
\end{proposition}
\begin{proof}
We want to show that $|\solproto| \lesssim |\zeta|^{\tau-1} e^{\lambda_0|\zeta|}$ for some real constant $\lambda_0$.

By Proposition~\ref{prop:asymptotic at zero}, there is a radius $\delta > 0$ for which $|\solproto| \lesssim |\zeta|^{\tau-1}$ over the region $|\zeta| < \delta$. No matter which value of $\lambda_0$ we pick, $e^{\lambda_0|\zeta|}$ cannot get arbitrarily close to zero on a bounded region, so $|\solproto| \lesssim |\zeta|^{\tau-1} e^{\lambda_0|\zeta|}$ over the region $|\zeta| < \delta$.

By Condition~\eqref{cond:diag-basic}, we have
\begin{align*}
|\hardker(a, a')| & \lesssim \frac{1}{|\zeta(a)|} e^{\lambda_\Delta |\zeta(a) - \zeta(a')|} \\
& \lesssim \delta^{-1} e^{\lambda_\Delta |\zeta(a) - \zeta(a')|} \\
& \lesssim e^{\lambda_\Delta |\zeta(a) - \zeta(a')|}
\end{align*}
over all $a, a' \in \domain$ with $|\zeta(a)| \ge \delta$. The last bound means that for some $c_\Delta > 0$,
\[ |\hardker(a, a')| \le c_\Delta e^{\lambda_\Delta |\zeta(a) - \zeta(a')|} \]
whenever $|\zeta(a)| \ge \delta$. Applying this bound along the diagonal in $\domain^2$, we conclude that $|q/p| \le c_\Delta$ on the region $|\zeta| \ge \delta$. On the other hand, by fixing some arbitrary  $b' \in \domain$, we see that
\begin{align*}
\left|\frac{q(b')}{p(a)}\right| & \lesssim \frac{1}{|\zeta(a)|} e^{\lambda_\Delta|\zeta(a) - \zeta(b')|} \\
& \lesssim \frac{1}{|\zeta(a)|} e^{\lambda_\Delta(|\zeta(a)| + |\zeta(b')|)} \\
& \lesssim \frac{1}{|\zeta(a)|} e^{\lambda_\Delta|\zeta(a)|} e^{\lambda_\Delta|\zeta(b')|} \\
& \lesssim \frac{1}{|\zeta(a)|} e^{\lambda_\Delta|\zeta(a)|}
\end{align*}
over all $a \in \domain$ with $|\zeta(a)| \ge \delta$. Using both of the bounds we just found, we deduce that
\begin{align*}
|\solproto(a)| & = \left| \frac{1}{p(a)} \exp\left(-\int_{b}^{a}\frac{q}{p}\;d\zeta\right) \right| \\
& \le \left|\frac{1}{p(a)}\right| \exp\left(\int_{b}^{a}\left|\frac{q}{p}\right|\;d\zeta\right) \\
& \lesssim \frac{1}{|\zeta(a)|} e^{\lambda_\Delta|\zeta(a)|}\,e^{c_\Delta(|\zeta(a)| + |\zeta(b)|)} \\
& \lesssim \frac{1}{|\zeta(a)|} e^{(\lambda_\Delta + c_\Delta)|\zeta(a)|}\,e^{c_\Delta|\zeta(b)|}
\end{align*}
over all $a \in \domain$ with $|\zeta(a)| \ge \delta$. Setting $\lambda_0 = \lambda_\Delta + c_\Delta$, we have $|\solproto| \lesssim |\zeta|^{-1} e^{\lambda_0|\zeta|}$ over the region $|\zeta| \ge \delta$. Since we are assuming $\tau$ is real and positive, $|\zeta|^\tau$ cannot get arbitrarily close to zero on the region $|\zeta| > \delta$. It follows that $|\solproto| \lesssim |\zeta|^{\tau-1} e^{\lambda_0|\zeta|}$ over the region $|\zeta| \ge \delta$. Combining this with our earlier argument on the region $|\zeta| < \delta$, we get the desired result.
\end{proof}
\begin{proposition}\label{prop:better-proto-estimate}
Suppose $\hardpart$ satisfies {\em Conditions~\eqref{cond:reg-p}} and \eqref{cond:sing}. Then there is some $M\in\C$ for which
\[ \solproto \in M\zeta^{\tau-1} + O(|\zeta|^{\tau-1+\epsilon'}) \]
at $\zeta = 0$, where $\epsilon'=\min\{\epsilon, 1\}$.
\end{proposition}
\begin{proof}
It is enough to show that
\[ \solproto \in M\zeta^{\tau-1} \Big[1 + O\big(|\zeta|^{\epsilon'}\big)\Big] \]
at $\zeta = 0$. As in the proof of Theorem~\ref{thm:basic_volterra}, we compute
\begin{align*}
\solproto(a)&=\frac{1}{p(a)}\, \exp\left(-\int_b^a\frac{q}{p} d\zeta\;\right)\\
& = \frac{1}{p(a)}\,\frac{\zeta(a)^{\tau}}{\zeta(b)^{\tau}}\,\exp\left[-\int_b^a \left(\frac{q}{p}+\frac{\tau}{\zeta}\right) d\zeta\right]\\
& = \zeta(b)^{-\tau}\,\zeta(a)^{\tau-1}\,\frac{\zeta(a)}{p(a)} \exp\left[-\int_b^a\left(\frac{q}{p}+\frac{\tau}{\zeta}\right) d\zeta\right].
\end{align*}
Note that $\zeta(b)^{-\tau}$ is a constant, which will eventually be absorbed into $M$.

Condition \eqref{cond:reg-p} implies that
\[ \frac{\zeta}{p} \in B^{-1} + O\big(|\zeta|^{\epsilon}\big), \]
at $\zeta = 0$.

Now, consider the exponential factor. By Condition~\eqref{cond:sing}, there exists a constant $C$ and a neighborhood $\domain_\text{near}$ of $\zeta = 0$ for which
\[ \left| \frac{q}{p}+\frac{\tau}{\zeta} \right| < C \]
in $\domain_\text{near}$. It follows that the improper integral
\[ \eta(a) = \int_{\zeta = 0}^a \left(\frac{q}{p}+\frac{\tau}{\zeta}\right) d\zeta \]
converges for all $a \in \domain$, allowing us to write
\[ \exp\left[-\int_b^a\left(\frac{q}{p}+\frac{\tau}{\zeta}\right) d\zeta\right] = e^{\eta(b)} e^{-\eta(a)}. \]
Observing that $|\eta| \le C |\zeta|$ in $\near$, we conclude that
\[ \exp\left[-\int_b^a\left(\frac{q}{p}+\frac{\tau}{\zeta}\right) d\zeta\right] \in e^{\eta(b)} \Big[ 1 + O\big(|\zeta(a)|\big) \Big] \]
at $\zeta(a) = 0$.

Together, the arguments above show that
\begin{align*}
\solproto & \in \zeta(b)^{-\tau}\,\zeta^{\tau-1}\;\Big[ B^{-1} + O\big(|\zeta|^{\epsilon}\big) \Big]\;e^{\eta(b)} \Big[1 + O\big(|\zeta|\big) \Big] \\
& = \zeta(b)^{-\tau} B^{-1} e^{\eta(b)}\;\zeta^{\tau-1} \Big[ 1 + O\big(|\zeta|^{\epsilon}\big) \Big] \Big[1 + O\big(|\zeta|\big) \Big] \\
& = M\zeta^{\tau-1} \Big[ 1 + O\big(|\zeta|^{\epsilon}\big) + O\big(|\zeta|\big) \Big]
\end{align*}
at $\zeta = 0$, with $M = \zeta(b)^{-\tau} B^{-1} e^{\eta(b)}$. Observing that $O\big(|\zeta|^{\epsilon}\big) + O\big(|\zeta|\big) = O\big(|\zeta|^{\epsilon'}\big)$, we reach the desired result.
\end{proof}
\subsection{Showing that $\softpart$ makes the prototype solution less singular \\ \textit{(toward Theorem~\ref{thm:general_volterra})}}\label{sec:image under soft_part}
Now that we know the prototype solution $\solproto$ belongs to $\singexpalg{\tau-1}(\domain)$, we will show that $\softpart$ reduces the sharpness of its singularity at $\zeta = 0$, mapping it into $\singexpalg{\tau-1+\gamma}(\domain)$. This is a consequence of the general smoothing effect of $\softpart$, described in the following result.

\begin{proposition}\label{prop:smoothing}
Under {\em Condition~\eqref{cond:eps-lambda}}, the operator $\softpart$ maps
\[ \singexp{\sigma}{\Lambda}(\Omega) \to \singexp{\sigma+\gamma}{\Lambda}(\Omega) \]
continuously for all $\Lambda\geq \lambda_{\Delta}$ and $\sigma>-1$.
\end{proposition}
\begin{rmk}
We assume $\gamma > 0$, but this result holds even under the weaker assumption that $\gamma > -1$. We will take advantage of this in Section~\ref{sec:L-int-op}.
\end{rmk}
\begin{proof}[Proof of Proposition~\ref{prop:smoothing}]
For any function $f\in\singexp{\sigma}{\Lambda}(\domain)$,
\begin{revtwo}
\begin{align*}
& |\zeta(a)|^{-(\sigma+\gamma)} \, e^{-\Lambda |\zeta(a)|} \, \Big \vert \big[ \softpart f\big](a)\Big\vert \\
& \qquad \leq |\zeta(a)|^{-(\sigma+\gamma)}\, e^{-\Lambda |\zeta(a)|} \int_{\zeta(a')=0}^a |\softker(a,a')|\, |f(a')| \, |d\zeta(a')| \\
& \qquad \leq |\zeta(a)|^{-(\sigma+\gamma)}\, e^{-\Lambda |\zeta(a)|} \int_{\zeta(a')=0}^a |\softker(a,a')|\, |\zeta(a')|^{\sigma}\, e^{\Lambda |\zeta(a')|} \|f\|_{\sigma,\Lambda} \, |d\zeta(a')|\,.
\end{align*}
\end{revtwo}
By Condition~\eqref{cond:eps-lambda},
\begin{revtwo}
\begin{align*}
& \int_{\zeta(a')=0}^a |\softker(a,a')|\, |\zeta(a')|^{\sigma}\, e^{\Lambda |\zeta(a')|} \, |d\zeta(a')| \\
& \qquad \lesssim \int_{\zeta(a')=0}^a \frac{|\zeta(a)-\zeta(a')|^\gamma}{|\zeta(a)|} e^{\lambda_\Delta |\zeta(a)-\zeta(a')|} |\zeta(a')|^{\sigma}\, e^{\Lambda|\zeta(a')|}\;|d\zeta(a')|\\
& \qquad \lesssim |\zeta(a)|^{\gamma+\sigma} \int_{0}^1 (1-t)^\gamma e^{\lambda_\Delta |\zeta(a)|(1-t)} t^{\sigma}\, e^{\Lambda|\zeta(a)| t}\;dt\\
& \qquad \lesssim |\zeta(a)|^{\gamma+\sigma}\, e^{\lambda_\Delta |\zeta(a)|}\,  \int_{0}^1 (1-t)^\gamma  t^{\sigma}\,e^{(\Lambda-\lambda_\Delta)|\zeta(a)| t}\;dt\\
& \qquad \lesssim |\zeta(a)|^{\gamma+\sigma}\, e^{\lambda_\Delta |\zeta(a)|}\,e^{(\Lambda-\lambda_\Delta)|\zeta(a)|}\int_{0}^1 (1-t)^\gamma  t^{\sigma}\;dt.
\end{align*}
\end{revtwo}
The last step takes advantage of the assumption that $\Lambda \ge \lambda_\Delta$. The integral on the last line is Euler's beta integral~\cite[equation~(5.12.1)]{dlmf}, so
\begin{revtwo}
\[ \int_{\zeta(a')=0}^a |\softker(a, a')|\, |\zeta(a')|^{\sigma}\, e^{\Lambda |\zeta(a')|} \, |d\zeta(a')| \lesssim |\zeta(a)|^{\gamma+\sigma}\, e^{\Lambda |\zeta(a)|} \frac{\Gamma(\gamma+1)\Gamma(\sigma+1)}{\Gamma(\sigma+\gamma+2)}\,. \]
\end{revtwo}
Since $\gamma > 0$ and $\sigma > -1$, the gamma functions are well-defined.\footnote{We could even weaken the constraint on $\gamma$, allowing any $\gamma > -1$.} Rearranging to get
\begin{revtwo}
\[ |\zeta(a)|^{-(\gamma+\sigma)}\, e^{-\Lambda |\zeta(a)|} \int_{\zeta(a')=0}^a |\softker(a,a')|\, |\zeta(a')|^{\sigma}\, e^{\Lambda |\zeta(a')|} \, |d\zeta(a')| \lesssim \frac{\Gamma(\gamma+1)\Gamma(\sigma+1)}{\Gamma(\sigma+\gamma+2)}\,, \]    
\end{revtwo}
we conclude that $|\zeta(a)|^{-\sigma-\gamma} \, e^{-\Lambda |\zeta(a)|} \, \Big \vert \big[ \softpart f\big](a)\Big\vert$ is uniformly bounded in $\domain$. 
\end{proof}
\begin{corollary}\label{cor:pertub_f0}
Consider the prototype solution $\solproto$ from equation~\eqref{eqn:test_solution}. If $\softpart$ satisfies {\em Condition \eqref{cond:eps-lambda}}, then $\softpart \solproto$ belongs to $\singexpalg{\tau-1+\gamma}(\domain)$.
\end{corollary}
\begin{proof}
We know from Theorem~\ref{thm:proto-growth} that $\solproto$ belongs to $\singexpalg{\tau-1}(\domain)$. Choose a constant $\Lambda \ge \lambda_\Delta$ big enough that $\solproto$ is in $\singexp{\tau-1}{\Lambda}(\domain)$. Since we are assuming $\tau$ is positive, we can apply Proposition~\ref{prop:smoothing}, concluding that $\softpart \solproto$ belongs to $\singexp{\tau-1+\gamma}{\Lambda}(\domain)$.
\end{proof}
\subsection{Showing that $\volterra$ shrinks less singular functions \\ \textit{(toward Lemma~\ref{lem:perturbed_volterra})}}\label{sec:V is a contraction}
\subsubsection{Overview}
In this section, we will prove the following proposition. 

\begin{proposition}\label{prop:get-contraction}
Suppose that $\hardpart$ satisfies {\em Conditions \eqref{cond:sing}} and \eqref{cond:diag-basic}, and $\softpart$ satisfies {\em Condition \eqref{cond:eps-lambda}}. Then, for each $\rho > \tau$, there is some $\Lambda_\text{\rm lower} \in \R$ such that
\[\volterra\colon\singexp{\rho-1}{\Lambda}(\domain)\to\singexp{\rho-1}{\Lambda}(\domain)\]
is a contraction for all $\Lambda > \Lambda_\text{\rm lower}$.
\end{proposition}
\begin{rmk}
The threshold $\Lambda_\text{\rm lower}$ is always greater than $\lambda_\Delta$, and often much greater.
\end{rmk}
First, pick some $\sigma \in (\tau, \rho)$. Proposition~\ref{prop:whole-ker-near-bound}, which we will state and prove in Section~\ref{sec:bounds on k}, shows that there is a neighborhood $\near \subset \domain$ of $\zeta = 0$ with the property that
\begin{equation}\label{near-limit}
|k(a, a')| \le \frac{\sigma}{|\zeta(a)|}
\end{equation}
for all $a, a' \in \near$. We take $\near$ to be the part of $\domain$ where $|\zeta| < \delta$, for some small enough positive radius $\delta$. Complementarily, let $\far$ be the part of $\domain$ where $\delta \le |\zeta|$.

Take any function $\varphi \in \singexp{\rho}{\Lambda}(\domain)$. In Section~\ref{near-bound}, we will bound $|\zeta|^{-(\rho-1)}\,e^{-\Lambda|\zeta|}\,|\volterra\varphi|$ by $\tfrac{\sigma}{\rho} \|\varphi\|_{\rho-1, \Lambda}$ on $\near$. In Section~\ref{far-bound}, we will see that by making $\Lambda$ big enough, we can bound $|\zeta|^{-(\rho-1)}\,e^{-\Lambda|\zeta|}\,|\volterra\varphi|$ by an arbitrarily small constant multiple of $\|\varphi\|_{\rho-1, \Lambda}$ on $\far$. Together, these results will show that $\|\volterra \varphi\|_{\rho-1, \Lambda} \le \tfrac{\sigma}{\rho} \|\varphi\|_{\rho-1, \Lambda}$ when $\Lambda$ is large enough. Since we set $\sigma < \rho$, this will prove Proposition~\ref{prop:get-contraction}.
\subsubsection{Bounds on the perturbed kernel}\label{sec:bounds on k}
The conditions on $\hardker$ and $\softker$ described in Sections \ref{setting:basic} and \ref{setting:perturbed} can be combined into convenient bounds on the kernel $\kerwhole = \hardker + \softker$ of $\volterra$. One bound, which works when both arguments of $\kerwhole$ are close to $\zeta = 0$, will be used in Section~\ref{near-bound}.
\begin{proposition}\label{prop:whole-ker-near-bound}
Suppose $\hardpart$ satisfies Condition~\eqref{cond:sing}, and $\softpart$ satisfies Condition~\eqref{cond:eps-lambda}. Then, for any $\sigma > \tau$, there is a neighborhood $\Sigma \subset \Omega$ of $\zeta = 0$ with the property that
\[ |\kerwhole(a, a')| \le \frac{\sigma}{|\zeta(a)|} \]
for all $a, a' \in \Sigma$.
\end{proposition}
\begin{proof}
Choose some $\sigma_0 \in (\tau, \sigma)$. From Condition~\eqref{cond:sing}, by making $\delta > 0$ small enough, we can ensure that
\[ |\hardker(a, a')| \le \frac{\sigma_0}{|\zeta(a)|} \]
whenever $|\zeta(a)| < \delta$ and $|\zeta(a')| < \delta$.

Condition~\eqref{cond:eps-lambda} implies that
\[ | \softker(a, a') | \lesssim\frac{|\zeta(a)-\zeta(a')|^\gamma}{|\zeta(a)|}\,e^{\lambda_\Delta|\zeta(a)-\zeta(a')|}\]
over all $a, a' \in \domain$, which means that
\[ | \softker(a, a') | \lesssim\frac{(2\delta)^\gamma}{|\zeta(a)|}\,e^{\lambda_\Delta(2\delta)}\]
over all $a, a' \in \domain$ with $|\zeta(a)| < \delta$ and $|\zeta(a')| < \delta$. Combining this conclusion with the previous one, we get the desired result.
\end{proof}
Another bound, which works when the first argument of $\kerwhole$ is kept away from $\zeta = 0$, will be used in Section~\ref{far-bound}.
\begin{proposition}\label{prop:whole-ker-far-bound}
Suppose $\hardpart$ satisfies {\em Condition~\eqref{cond:diag-basic}}, and $\softpart$ satisfies {\em Condition~\eqref{cond:eps-lambda}}. Choose a subset $\far \subset \domain$ that does not touch $\zeta = 0$. Then, for any $\lambda > \lambda_\Delta$, we have
\[ |\kerwhole(a,a')| \lesssim e^{\lambda |\zeta(a)-\zeta(a')|} \]
over all $a \in \far$ and $a' \in \domain$.
\end{proposition}
\begin{proof}
Find a radius $\delta > 0$ with $|\zeta| \ge \delta$ on $\far$, and choose some $\lambda > \lambda_\Delta$. Conditions \eqref{cond:diag-basic} and \eqref{cond:eps-lambda} tell us that
\begin{align*}
|\kerwhole(a,a')|&\leq |\hardker(a,a')| + |\softker(a,a')|\\
&\lesssim \frac{1}{|\zeta(a)|} e^{\lambda_\Delta |\zeta(a)-\zeta(a')|} + \frac{|\zeta(a)-\zeta(a')|^\gamma}{|\zeta(a)|}\,e^{\lambda_\Delta|\zeta(a)-\zeta(a')|}\\
&\lesssim \delta^{-1} \big(1 + |\zeta(a)-\zeta(a')|^\gamma \big) \, e^{\lambda_\Delta|\zeta(a)-\zeta(a')|}
\end{align*}
over all $a \in \far$ and $a' \in \domain$. Since
\[ 1 + |\zeta(a)-\zeta(a')|^\gamma \]
grows polynomially with respect to $|\zeta(a)-\zeta(a')|$, we can bound it with any growing exponential function of $|\zeta(a)-\zeta(a')|$. In particular,
\[ 1 + |\zeta(a)-\zeta(a')|^\gamma \lesssim e^{(\lambda - \lambda_\Delta) |\zeta(a)-\zeta(a')|} \]
over all $a, a' \in \domain$. It follows that
\[ |\kerwhole(a,a')| \lesssim e^{(\lambda - \lambda_\Delta) |\zeta(a)-\zeta(a')|} \, e^{\lambda_\Delta|\zeta(a)-\zeta(a')|} \]
over all $a \in \far$ and $a' \in \domain$. This simplifies to the desired result.
\end{proof}
\subsubsection{First steps toward showing that $\volterra$ is a contraction}\label{first-steps}
The first steps of our calculation are the same throughout $\domain$. For each $a \in \domain$, we have
\begin{revtwo}
\[ |\zeta(a)|^{-(\rho-1)}\,e^{-\Lambda|\zeta(a)|}\,\big|[\volterra\varphi](a)\big| \le |\zeta(a)|^{-(\rho-1)}\,e^{-\Lambda|\zeta(a)|}  \int_{\zeta(a') = 0}^a |k(a, a')\,\varphi(a')\;d\zeta(a')| \]
\end{revtwo}
for any choice of integration path. The norm on $\singexp{\rho-1}{\Lambda}(\domain)$ is designed to give the bound $|\varphi| \le |\zeta|^{\rho-1}\,e^{\Lambda |\zeta|}\,\|\varphi\|_{\rho-1, \Lambda}$, so
\begin{revtwo}
\begin{align*}
&|\zeta(a)|^{-(\rho-1)}\,e^{-\Lambda|\zeta(a)|}\,\big|[\volterra\varphi](a)\big| \\
& \qquad \le |\zeta(a)|^{-(\rho-1)}\,e^{-\Lambda|\zeta(a)|}  \int_{\zeta(a') = 0}^a |k(a, a')|\,|\zeta(a')|^{\rho-1}\,e^{\Lambda |\zeta(a')|}\,\|\varphi\|_{\rho-1, \Lambda}\;|d\zeta(a')| \\
& \qquad = \|\varphi\|_{\rho-1, \Lambda}  \int_{\zeta(a') = 0}^a |k(a, a')|\,\left|\frac{\zeta(a')}{\zeta(a)}\right|^{\rho-1}\,e^{-\Lambda(|\zeta(a)| - |\zeta(a')|)}\;|d\zeta(a')|
\end{align*}
\end{revtwo}
What we do next depends on whether $a$ is in $\near$ or $\far$.
\subsubsection{Near the origin}\label{near-bound}
Suppose that $a \in \near$. Then inequality~\eqref{near-limit} gives
\begin{align*}
|\zeta(a)|^{-(\rho-1)}\,e^{-\Lambda|\zeta(a)|}\,\big|[\volterra\varphi](a)\big| & \le
\|\varphi\|_{\rho-1, \Lambda}  \int_{\zeta = 0}^a \frac{\sigma}{|\zeta(a)|}\,\left|\frac{\zeta}{\zeta(a)}\right|^{\rho-1}\,e^{-\Lambda(|\zeta(a)| - |\zeta|)}\;|d\zeta| \\
& \le \sigma \|\varphi\|_{\rho-1, \Lambda}  \int_{\zeta = 0}^a \left|\frac{\zeta}{\zeta(a)}\right|^{\rho-1}\,e^{-\Lambda(|\zeta(a)| - |\zeta|)}\;\left|\frac{d\zeta}{\zeta(a)}\right|
\end{align*}
for any integration path that stays within $\near$. Noting that $\near$ is star-shaped by construction, we use the straight path $\zeta = t \zeta(a)$, with $t \in (0, 1]$:
\begin{align*}
|\zeta(a)|^{-(\rho-1)}\,e^{-\Lambda|\zeta(a)|}\,\big|[\volterra\varphi](a)\big| & \le \sigma \|\varphi\|_{\rho-1, \Lambda} \int_0^1 t^{\rho-1}\,e^{-\Lambda |\zeta(a)|(1 - t)}\;dt \\
& \le \sigma \|\varphi\|_{\rho-1, \Lambda} \int_0^1 t^{\rho-1}\;dt \\
& = \frac{\sigma}{\rho} \|\varphi\|_{\rho-1, \Lambda}.
\end{align*}
Since we set $\sigma < \rho$, this brings us halfway to proving Proposition~\ref{prop:get-contraction}.
\subsubsection{Away from the origin}\label{far-bound}
Going back to the end of Section~\ref{first-steps}, suppose that $a \in \far$. Choose some $\lambda > \lambda_\Delta$. By Proposition~\ref{prop:whole-ker-far-bound} from Section~\ref{sec:bounds on k}, there is some $M\in\R_{>0}$ for which
\[ |k(a, a')| \le M e^{\lambda |\zeta(a) - \zeta(a')|} \]
for all $a \in \far$ and $a' \in \Omega$. This bound implies that
\[ |\zeta(a)|^{-(\rho-1)}\,e^{-\Lambda|\zeta(a)|}\,\big|[\volterra\varphi](a)\big| \le \|\varphi\|_{\rho-1, \Lambda} \int_{\zeta = 0}^a M e^{\lambda |\zeta(a) - \zeta|}\,\left|\frac{\zeta}{\zeta(a)}\right|^{\rho-1}\,e^{-\Lambda(|\zeta(a)| - |\zeta|)}\;|d\zeta|. \]
We again use the straight integration path $\zeta = t \zeta(a)$, with $t \in (0, 1]$. Along this path, $|\zeta(a) - \zeta| = |\zeta(a)| - |\zeta|$, allowing us to combine the exponential factors in our bound:
\begin{align*}
|\zeta(a)|^{-(\rho-1)}\,e^{-\Lambda|\zeta(a)|}\,\big|[\volterra\varphi](a)\big| & \le \|\varphi\|_{\rho-1, \Lambda} \int_0^1 M e^{\lambda |\zeta(a)|(1 - t)}\,t^{\rho-1}\,e^{-\Lambda |\zeta(a)|(1 - t)}\;dt \\
& \le M \|\varphi\|_{\rho-1, \Lambda} \int_0^1 t^{\rho-1}\,e^{-(\Lambda - \lambda)|\zeta(a)|(1 - t)}\;dt.
\end{align*}
Set $\Lambda > \lambda$, so the exponential factor shrinks as $|\zeta(a)|$ grows. Since $|\zeta(a)| \ge \delta$ for all $a \in \far$, we then have the uniform bound
\[ |\zeta(a)|^{-(\rho-1)}\,e^{-\Lambda|\zeta(a)|}\,\big|[\volterra\varphi](a)\big| \le M \|\varphi\|_{\rho-1, \Lambda} \int_0^1 t^{\rho-1}\,e^{-(\Lambda - \lambda)\delta(1 - t)}\;dt. \]
over all $a \in \far$.

We can make this bound less than one, as required to show that $\volterra$ is a contraction, by increasing $\Lambda$. To see this in the trickiest case, where $\rho < 1$, consider the beginning and end of the integration path separately. At the beginning of the path---for $t \in \big(0, \tfrac{1}{5}\big]$, say---we use a worst-case estimate on the exponential factor:
\begin{align*}
\int_0^{1/5} t^{\rho-1}\,e^{-(\Lambda - \lambda)\delta(1 - t)}\;dt & \le e^{-\frac{4}{5}(\Lambda -\lambda)\delta} \int_0^{1/5} t^{\rho-1}\;dt \\
& = e^{-\frac{4}{5}(\Lambda - \lambda)\delta}\,\tfrac{1}{\rho} \big(\tfrac{1}{5}\big)^\rho.
\end{align*}
At the end of the path, we instead use a worst-case estimate on $t^{\rho-1}$:
\begin{align*}
\int_{1/5}^1 t^{\rho-1}\,e^{-(\Lambda - \lambda)\delta(1 - t)}\;dt & \le \max\{(1/5)^{\rho-1}, 1\} \int_{1/5}^1 e^{-(\Lambda - \lambda)\delta(1 - t)}\;dt \\
& = \max\{(1/5)^{\rho-1}, 1\} \frac{1}{(\Lambda - \lambda)\delta}\left[ 1 - e^{-\tfrac{4}{5}(\Lambda - \lambda)\delta} \right] \\
& \le \max\{(1/5)^{\rho-1}, 1\} \frac{1}{(\Lambda - \lambda)\delta}.
\end{align*}
In summary, the beginning of the integral is bounded by a decaying exponential function of $(\Lambda - \lambda)\delta$, and the end of the integral is bounded by a reciprocal function of $(\Lambda - \lambda)\delta$. That means we can make $|\zeta(a)|^{-(\rho-1)}\,e^{-\Lambda|\zeta(a)|}\,\big|[\volterra\varphi](a)\big|$ as small as we want over all $a \in \far$. This completes our proof of Proposition~\ref{prop:get-contraction}.
\subsection{Existence and uniqueness of a fixed point \\ \textit{(proof of Lemma~\ref{lem:perturbed_volterra} and Theorem~\ref{thm:general_volterra})}}\label{sec:existence and uniqueness}
\begin{proof}[Proof of Lemma~\ref{lem:perturbed_volterra}]
Choose $\Lambda \in \R$ large enough that $\singexp{\rho-1}{\Lambda}(\Omega)$ contains $g$ and is contracted by $\volterra$. Proposition~\ref{prop:get-contraction} and Definition~\ref{def:exp-top} guarantee that this is possible. The affine map $f \mapsto \volterra f + g$ is a contraction of $\singexp{\rho-1}{\Lambda}(\Omega)$, and thus has a unique fixed point in $\singexp{\rho-1}{\Lambda}(\Omega)$ by the contraction mapping theorem.

To see that the fixed point is unique in the larger space $\singexpalg{\rho-1}(\Omega)$, first recall from Proposition~\ref{prop:inclus-ge-exp} that we have inclusions $\singexp{\rho-1}{\Lambda'}(\Omega) \hookrightarrow \singexp{\rho-1}{\Lambda}(\Omega)$ for all $\Lambda' \le \Lambda$. Any fixed point in $\singexp{\rho-1}{\Lambda'}(\Omega)$ must map to the unique fixed point in $\singexp{\rho-1}{\Lambda}(\Omega)$ under this inclusion. Next, observe that for any $\Lambda'' \ge \Lambda$, the map $f \mapsto \volterra f + g$ is also a contraction of $\singexp{\rho-1}{\Lambda''}(\Omega)$. The inclusion $\singexp{\rho-1}{\Lambda}(\Omega) \hookrightarrow \singexp{\rho-1}{\Lambda''}(\Omega)$ must send the unique fixed point in the smaller space to the unique fixed point in the larger one. Together, these arguments show that the fixed point is unique in $\singexpalg{\rho-1}(\Omega)$.
\end{proof}
\begin{proof}[Proof of Theorem~\ref{thm:general_volterra}]
Start with the prototype solution $\solproto$ constructed in Theorem~\ref{thm:basic_volterra}. The base point for the construction can be chosen arbitrarily. Under our assumptions about $\volterra$, Theorem~\ref{thm:proto-growth} ensures that $\solproto$ is in $\singexpalg{\tau-1}(\domain)$. Our goal is to find a perturbation $\solptb \in \singexpalg{\tau-1+\gamma}(\domain)$ that makes $\solwhole = \solproto + \solptb$ a solution of 
\begin{equation}\label{eqn:orig-homog}
\solwhole = \volterra \solwhole.
\end{equation}
Observing that
\begin{align*}
\volterra \solproto & = \hardpart\solproto + \softpart\solproto \\
& = \solproto + \softpart \solproto,
\end{align*}
we rewrite the homogeneous equation we are trying to solve as an inhomogeneous equation for $\solptb$:
\begin{align}
\solproto + \solptb & = \volterra\solproto + \volterra\solptb \nonumber \\
& = \solproto + \softpart\solproto + \volterra\solptb \nonumber \\
\solptb & = \softpart\solproto + \volterra\solptb. \label{eqn:equiv-inhomog}
\end{align}
Since we know that $\solproto$ is in $\singexpalg{\tau-1}(\domain)$, Proposition~\ref{prop:smoothing} shows that the inhomogeneous term $\softpart\solproto$ is in $\singexpalg{\tau-1+\gamma}(\domain)$. Since $\tau+\gamma > \tau$, and we have made all the necessary assumptions about $\volterra$, Lemma~\ref{lem:perturbed_volterra} guarantees that equation~\eqref{eqn:equiv-inhomog} has a unique solution $\solptb$ in $\singexpalg{\tau-1+\gamma}(\Omega)$. Equivalently, equation~\eqref{eqn:orig-homog} has a unique solution $f$ in $f_0 + \singexpalg{\tau-1+\gamma}(\domain)$.
\end{proof}
\begin{proof}[Proof of Corollary~\ref{cor:expand_uniq}]
First, suppose $\rho \in (\tau, \tau+\gamma)$. In this case, the inclusion
\[\singexpalg{\tau-1+\gamma}(\Omega) \hookrightarrow \singexpalg{\rho-1}(\Omega)\]
given by Proposition~\ref{prop:inclus-lt-pow-alg} shows that the inhomogeneous term $\softpart\solproto$ is in $\singexpalg{\rho-1}(\Omega)$. Lemma~\ref{lem:perturbed_volterra} then guarantees that equation~\eqref{eqn:equiv-inhomog} has a unique solution in $\singexpalg{\rho-1}(\Omega)$---which must be the solution we already found in $\singexpalg{\tau-1+\gamma}(\Omega)$.

On the other hand, suppose $\rho > \tau+\gamma$. In this case, Proposition~\ref{prop:inclus-lt-pow-alg} gives an inclusion $\singexpalg{\rho-1}(\Omega) \hookrightarrow \singexpalg{\tau-1+\gamma}(\Omega)$. Under this inclusion, any solution of equation~\eqref{eqn:equiv-inhomog} that we might find in $\singexpalg{\rho-1}(\Omega)$ must match the unique solution we found in $\singexpalg{\tau-1+\gamma}(\Omega)$.
\end{proof}
\begin{proposition}\label{prop:alt-general_volterra}
Suppose $\volterra$ satisfies {\em Condition~\eqref{cond:reg-p}} in addition to the other assumptions of {\em Theorem~\ref{thm:general_volterra}}. Then, for each $\rho \in (\tau,\;\tau + \min\{\gamma, \epsilon, 1\}]$, the equation
\[f = \volterra f\]
has a unique solution $f$ in the affine subspace
\[ \zeta^{\tau-1}+\singexpalg{\rho-1}(\domain) \]
of the space $\singexpalg{\tau-1}(\domain)$.
\end{proposition}
\begin{proof}
Propositions \ref{prop:asymptotic at infinity} and \ref{prop:better-proto-estimate} imply that for some constant $M$,
\[ \zeta^{\tau-1} + \singexpalg{\rho-1}(\domain) = M^{-1}\solproto + \singexpalg{\rho-1}(\domain) \]
whenever $\rho \le \tau + \min\{\epsilon, 1\}$. Theorem~\ref{thm:general_volterra} implies that our equation has a unique solution in
\[ M^{-1}\solproto + \singexpalg{\rho-1}(\domain) \]
as long as $\rho \in (\tau, \tau + \gamma]$. Putting these facts together, we get the desired result.
\end{proof}
\section{A motivating example}\label{sec:example}
\subsection{Overview}
We can use our analysis of regular singular Volterra equations to study so-called \textit{level~1 differential equations}, which each have an irregular singularity at $\infty$. Under certain conditions, we can build a set of analytic solutions in the frequency domain by solving certain regular singular Volterra equations in the position domain. Our main results guarantee that these position domain solutions exist, are unique, and have well-defined Laplace transforms.
\subsection{Level $1$ differential equations}\label{sec:level 1 ODE}
Let $\mathcal{P}$ be a linear differential operator of the form
\[ \mathcal{P} = P\big(\tfrac{\partial}{\partial z}\big) + \frac{1}{z} Q\big(\tfrac{\partial}{\partial z}\big) + \frac{1}{z^2} R(z^{-1}), \]
where
\begin{enumerate}
\item[$\bullet$] $P$ is a monic degree-$d$ polynomial whose roots are all simple; 
\item[$\bullet$] $Q$ is a degree-$(d-1)$ polynomial that is nonzero at every root of $P$;
\item[$\bullet$] $R(z^{-1})$ is holomorphic in some disk $|z| > A$ around $z = \infty$. In particular, the power series
\[ R(z^{-1}) = \sum_{j=0}^\infty R_j z^{-j} \]
converges in the region $|z| > A$.
\end{enumerate}
Equations of the form $\mathcal{P}\Phi = 0$ are a sub-class of what Ecalle calls {\em level~1} differential equations~\cite[Section~2.1]{EcalleIII}\cite[Section~5.2.2.1]{diverg-resurg-iii}. In~\cite{borel_reg}, we study various examples of such equations using Laplace transform methods, with the help of our existence and uniqueness result Theorem~\ref{thm:general_volterra}. 
\subsection{Notation}\label{sec:notation_alpha}
So far, we have studied a Volterra equation with a regular singularity at $\zeta = 0$. When we use Laplace transform methods to solve a level~1 differential equation $\mathcal{P}\Phi = 0$ in the frequency domain, we will end up with several Volterra equations $\hat{\mathcal{P}}_\alpha \varphi = 0$ in the position domain, each with a regular singularity at a different root $\zeta = \alpha$ of the polynomial $P(-\zeta)$.

To use our previous reasoning in this more general setting, we must reinterpret its language. The role of the coordinate $\zeta$ in Sections \ref{sec:intro}--\ref{sec:proof_main_results} will be played by $\zeta-\alpha$ in Section~\ref{sec:example}. This substitution leads to several other reinterpretations, summarized below.

References to the point $\zeta=0$ become references to the point $\zeta-\alpha=0$, which we will rewrite as $\zeta=\alpha$. For example, we will now work on a domain that touches but does not contain $\zeta = \alpha$, introducing the notation $\domain_\alpha$ as a reminder of this change. Condition~\eqref{cond:star} now says that $\domain_\alpha$ is star-shaped around $\zeta = \alpha$. The function space $\singexp{\sigma}{\Lambda}(\domain_\alpha)$ becomes a space of holomorphic functions on $\domain_\alpha$ which have a power-law singularity at $\zeta=\alpha$ and are uniformly of exponential type $\Lambda$. Explicitly, $f$ is in $\singexp{\sigma}{\Lambda}(\domain_\alpha)$ if 
\[ |f| \lesssim |\zeta-\alpha|^\sigma e^{\Lambda |\zeta-\alpha|} \]
over $\domain_\alpha$, making the norm
\[ \|f\|_{\sigma, \Lambda} = \sup_{\domain_\alpha} |\zeta-\alpha|^{-\sigma} e^{-\Lambda |\zeta-\alpha|} |f| \]
well-defined.

Our conditions on Volterra operators change to describe a regular singularity at $\zeta = \alpha$. For example, consider a Volterra operator $\hardpart^\alpha$ of the kind described in Section~\ref{setting:basic}, with kernel $\hardker^\alpha$. This operator now satisfies Condition~\eqref{cond:sing} if for some real, positive constant $\tau_\alpha$, the difference
\[ \hardker^\alpha(a, a) - \frac{\tau_\alpha}{\zeta(a) - \alpha} \]
is bounded on a neighborhood of $\zeta(a) = \alpha$ in $\domain_\alpha$, and for each $\sigma > \tau_\alpha$, the bound
\[ |\hardker^\alpha(a, a')| < \frac{\sigma}{|\zeta(a) - \alpha|} \]
holds over some neighborhood of $\big(\zeta(a), \zeta(a')\big) = (\alpha,\alpha)$ in $\domain_\alpha^2$. Condition~\eqref{cond:diag-basic} now requires that for some constant $\lambda_\Delta$,
\[|\hardker^\alpha(a,a')|\lesssim\frac{1}{|\zeta(a)-\alpha|} e^{\lambda_\Delta|\zeta(a)-\zeta(a')|}\]
over all $a,a'\in\domain_\alpha$. 
Similarly, an operator $\softpart^\alpha$ with kernel $\softker^\alpha$ now satisfies Condition~\eqref{cond:eps-lambda} if for some constants $\gamma > 0$ and $\lambda_\Delta$, we have
\[ |\softker^\alpha(a,a')| \lesssim \frac{|\zeta(a)-\zeta(a')|^\gamma}{|\zeta(a)-\alpha|} e^{\lambda_\Delta |\zeta(a)-\zeta(a')|}\]
over all $a, a' \in \domain_\alpha$.
\subsection{The Laplace transform}
\subsubsection{Definition}\label{sec:definition_Laplace}
Let $\domain_\alpha$ be a simply connected open subset of $\C$ that touches but does not contain $\zeta=\alpha$. Let $\Gamma_{\zeta, \alpha}^\theta$ be the ray in the position domain that leaves $\zeta=\alpha$ at angle $\theta$, and let $\Pi^{-\theta}_\Lambda$ be the half-plane $\Re(z e^{i\theta}) > \Lambda$ in the frequency domain. When $\domain_\alpha$ contains $\Gamma_{\zeta, \alpha}^\theta$, the Laplace transform
\[ \laplace_{\zeta, \alpha}^{\theta} \maps \singexp{\sigma}{\Lambda}(\domain_\alpha) \to \holo(\Pi^{-\theta}_\Lambda) \]
is defined for all $\sigma > -1$ and $\Lambda \in \R$. Its action on a function $\varphi$ is given by the formula
\begin{equation}\label{laplace:int}
\laplace_{\zeta, \alpha}^{\theta} \varphi := \int_{\Gamma_{\zeta,\alpha}^\theta} e^{-z\zeta} \varphi\;d\zeta.
\end{equation}

\subsubsection{Action on integral operators}\label{sec:L-int-op}
For each $\nu \in (0, \infty)$, the fractional integral $\partial^{-\nu}_{\zeta, \alpha}$ is the Volterra operator defined by
\[ [\partial^{-\nu}_{\zeta, \alpha} \varphi](a) := \frac{1}{\Gamma(\nu)} \int_{\zeta = \alpha}^a \big(\zeta(a)-\zeta\big)^{\nu-1} \varphi\;d\zeta \]
for each $a \in \domain_\alpha$. It obeys the semigroup law \cite[Section~1.3]{mladenov2014advanced}
\begin{align*}
\fracderiv{-\mu}{\zeta}{\alpha}\,\fracderiv{-\nu}{\zeta}{\alpha} & = \fracderiv{-\mu-\nu}{\zeta}{\alpha} & \mu, \nu \in (0, \infty),
\end{align*}
and agrees with ordinary repeated integration when $\nu$ is an integer \cite[equation~(35)]{mladenov2014advanced}.

Fractional integration has a smoothing effect: it reduces the sharpness of any power-law singularity at $\zeta = \alpha$.
\begin{proposition}\label{prop:frac-int-smoothing}
For each $\nu \in (0, \infty)$, the fractional integral $\fracderiv{-\nu}{\zeta}{\alpha}$ maps
\[ \singexp{\sigma}{\Lambda}(\domain_\alpha) \to \singexp{\sigma+\nu}{\Lambda}(\domain_\alpha) \]
continuously for all $\sigma > -1$ and $\Lambda \ge 0$.
\end{proposition}
\begin{proof}
Rewrite the fractional integral as
\begin{align*}
\left[\fracderiv{-\nu}{\zeta}{\alpha} \varphi\right](a)&=\frac{1}{\Gamma(\nu)}\int_{\zeta=\alpha}^a (\zeta(a)-\zeta)^{\nu-1} \, \varphi \; d\zeta\\
&=\frac{\zeta(a)-\alpha}{\Gamma(\nu)}\int_{\zeta=\alpha}^a \frac{(\zeta(a)-\zeta)^{\nu-1}}{\zeta(a)- \alpha}\,\varphi \; d\zeta.
\end{align*}
The Volterra operator with kernel
\[ h(a, a') = \frac{1}{\Gamma(\nu)}\,\frac{(\zeta(a)-\zeta(a'))^{\nu-1}}{\zeta(a)- \alpha} \]
satisfies Condition~\eqref{cond:eps-lambda} with $\gamma=\nu-1$ and $\lambda_\Delta=0$, as long as we loosen the condition to allow any $\gamma > -1$. Hence, by Proposition~\ref{prop:smoothing}, this operator maps
\[ \singexp{\sigma}{\Lambda}(\domain_\alpha) \to \singexp{\sigma+\nu-1}{\Lambda}(\domain_\alpha) \]
continuously for all $\sigma > -1$ and $\Lambda \ge 0$. Multiplication by $\zeta - \alpha$ then maps
\[ \singexp{\sigma+\nu-1}{\Lambda}(\domain_\alpha) \to \singexp{\sigma+\nu}{\Lambda}(\domain_\alpha) \]
continuously.
\end{proof}

From Proposition~\ref{prop:frac-int-smoothing}, we deduce that when $\varphi$ belongs to $\singexpalg{\sigma}(\domain_\alpha)$ for some $\sigma > -1$, its fractional integral $\fracderiv{-\nu}{\zeta}{\alpha}\varphi$ has a well-defined Laplace transform along any ray $\Gamma_{\zeta,\alpha}^{\theta}\subset\domain_\alpha$. Using a 2d integration argument, akin to the one in \cite[Theorem~2.39]{laplace-tfm}, one can show that 
\[ \laplace_{\zeta,\alpha}^{\theta}\,\fracderiv{-\nu}{\zeta}{\alpha} \varphi = z^{-\nu} \laplace_{\zeta, \alpha}^{\theta} \varphi \]
for all $\nu \in (0, \infty)$ and $\varphi\in\singexp{\sigma}{\Lambda}(\domain_\alpha)$ with $\sigma>-1$. This mirrors the Laplace transform's action on multiplication operators: under the same conditions on $\varphi$, differentiation under the integral shows that~\cite[Theorem~1.34]{laplace-tfm}
\[ \laplace_{\zeta,\alpha}^\theta (\zeta^n \varphi) = \big({-\tfrac{\partial}{\partial z}}\big)^n \laplace_{\zeta,\alpha}^\theta \varphi \]
for all integers $n \ge 0$ and suitable directions $\theta$.

\begin{rmk}
In the relationship between $\fracderiv{-\nu}{\zeta}{\alpha}$ and multiplication by $z^{-\nu}$ described above, smoothing out the singularity at $\zeta = \alpha$ in the position domain corresponds to speeding up the decay as $|z| \to \infty$ in the frequency domain. This is a general feature of the Laplace transform.
\end{rmk}
\subsection{Going to the position domain}
Under the Laplace transform $\laplace_{\zeta,\alpha}^{\theta}$, differential operators on the frequency domain pull back to integral operators on the position domain.
\begin{lemma}\label{lem:use-dict}
For any $\alpha \in \C$, let $\hat{\mathcal{P}}_{\alpha}$ be the integral operator 
\[ \hat{\mathcal{P}}_\alpha:=P(-\zeta)+\partial_{\zeta,\alpha}^{-1}\circ Q(-\zeta)+\partial_{\zeta,\alpha}^{-2}\circ R(\partial_{\zeta,\alpha}^{-1}). \]
If $\psi_\alpha$ satisfies the equation $\hat{\mathcal{P}}_\alpha\psi_\alpha=0$, then its Laplace transform $\Psi_\alpha:=\laplace_{\zeta,\alpha}^{\theta}\psi_\alpha$ satisfies the equation $\mathcal{P}\Psi_\alpha=0$, as long as the Laplace transform is well-defined.
\end{lemma}
\begin{rmk}
For the Laplace transform to be well-defined, $\domain_\alpha$ must contain the ray $\Gamma_{\zeta, \alpha}^\theta$ as a necessary condition. We will see later that for our results to apply, $\domain_\alpha$ also cannot touch any zero of $P(-\zeta)$ other than $\zeta = \alpha$. As a result, $\domain_\alpha$ might look like the domain illustrated in Section \ref{setting:domain}, reinterpreting the origin as $\zeta = \alpha$. 
\end{rmk}
\begin{rmk}\label{rmk:use-dict-explicit}
To write $\hat{\mathcal{P}}_\alpha$ more explicitly, let $p := P(-\zeta)$ and $q := Q(-\zeta)$, noting that $\zeta = \alpha$ is a root of $p$. Then
\begin{revtwo}
\[ \big[\hat{\mathcal{P}}_\alpha\varphi\big](a) = \big[ p\, \varphi \big](a) + \int_{\zeta(a')=\alpha}^a q(a') \, \varphi(a') \; d\zeta(a') + \int_{\zeta(a')=\alpha}^a k_R(a, a') \, \varphi(a') \, d\zeta(a'), \]
\end{revtwo}
where
\[ k_R(a,a') := \sum_{j=0}^\infty \frac{R_{j}}{(j+1)!} \, \big(\zeta(a)-\zeta(a')\big)^{j+1} . \]
\end{rmk}
\begin{proof}[Proof of Lemma~\ref{lem:use-dict}]
Comparing the definitions of $\mathcal{P}$ and $\hat{\mathcal{P}}_\alpha$, and using the properties of the Laplace transform discussed in Section~\ref{sec:L-int-op}, we find that $\mathcal{P} \laplace_{\zeta,\alpha}^\theta = \laplace_{\zeta,\alpha}^\theta \hat{\mathcal{P}}_\alpha$. Hence,
\begin{align*}
\mathcal{P}\Psi_\alpha & = \mathcal{P}\laplace_{\zeta,\alpha}^{\theta}\psi_\alpha \\
& = \laplace_{\zeta,\alpha}^{\theta}\hat{\mathcal{P}}_\alpha\psi_\alpha.
\end{align*}
\end{proof}

We can now state and prove the main result of this section. Choose a simply connected open set $\domain_\alpha$ that touches $\zeta = \alpha$ and does not touch any other roots of $P(-\zeta)$, as illustrated in Section~\ref{setting:domain} under the reinterpretation from Section \ref{sec:notation_alpha}. We will show that the Volterra equation $\hat{\mathcal{P}}_\alpha \varphi = 0$ has a unique solution of a certain form on $\domain_\alpha$. By Lemma~\ref{lem:use-dict}, the Laplace transform of this solution satisfies the equation $\mathcal{P}\Phi = 0$.
\begin{theorem}\label{thm:example}
Suppose $P$ has a root $-\alpha$ where $\tau_\alpha := Q(-\alpha)/P'(-\alpha)$ is real and positive. Choose a simply connected open set $\domain_\alpha$ that touches
$\zeta = \alpha$, and does not touch any other root of $P(-\zeta)$. Then the equation
\[ \hat{\mathcal{P}}_\alpha \psi_\alpha = 0 \]
has a unique solution $\psi_\alpha$ in the affine subspace
\[ \zeta^{\tau_\alpha-1} + \singexpalg{\tau_\alpha}(\domain_\alpha) \]
of the space $\singexpalg{\tau_\alpha-1}(\domain_\alpha)$.
\end{theorem}
\begin{proof}
To apply Proposition~\ref{prop:alt-general_volterra}, we must rewrite the equation $\hat{\mathcal{P}}_{\alpha}\varphi = 0$ in the form $\varphi = \volterra^\alpha \varphi$, where $\volterra^\alpha$ is a Volterra operator of the kind described in Section~\ref{setting} under the reinterpretation from Section~\ref{sec:notation_alpha}. When $\hat{\mathcal{P}}_\alpha$ is written as in Remark~\ref{rmk:use-dict-explicit}, the equation $\hat{\mathcal{P}}_{\alpha} \varphi = 0$ becomes
\begin{revtwo}
\[ \big[ p\, \varphi \big](a) + \int_{\zeta(a')=\alpha}^a q(a') \, \varphi(a') \; d\zeta(a') + \int_{\zeta(a')=\alpha}^a k_R(a, a') \, \varphi(a') \, d\zeta(a') = 0. \]
\end{revtwo}
Isolating $p\,\varphi$ on the left and dividing both sides by $p$ yields the equivalent equation
\begin{revtwo}
\begin{equation}\label{eqn:divided}
\varphi(a) = -\int_{\zeta(a')=\alpha}^a \frac{q(a')}{p(a)} \, \varphi(a') \; d\zeta(a') - \int_{\zeta(a')=\alpha}^a \frac{k_R(a, a')}{p(a)} \, \varphi(a') \, d\zeta(a').
\end{equation}
\end{revtwo}
Let $\hardpart^\alpha$ be the Volterra operator with kernel
\[ \hardker^\alpha(a,a') = -\frac{q(a')}{p(a)}, \]
and let $\softpart^\alpha$ be the Volterra operator with kernel
\[ \softker^\alpha(a,a') = -\frac{k_R(a,a')}{p(a)}. \]
Equation~\eqref{eqn:divided} can be written as $\varphi = \volterra^\alpha \varphi$, where $\volterra^\alpha = \hardpart^\alpha + \softpart^\alpha$.

The kernel of $\volterra^\alpha$ extends meromorphically over all of $\C^2$. To see this, recall that $R(z^{-1})$ is holomorphic on the disk $|z| > A$. It follows that for any $\lambda > A$, we have $|R_j| \lesssim \lambda^j$ over all $j \in \{0, 1, 2, \ldots\}$. The power series defining $k_R$ therefore converges everywhere.
\begin{verify}
\begin{align*}
& \sum_{j = 0}^\infty R_j z^{-j} \text{ converges when } |z| > A \\
\Longrightarrow & \sum_{j = 0}^\infty R_j z^{-j} \text{ converges when } |z^{-1}| < A^{-1} \\
\Longrightarrow & A^{-1} \le \operatorname{radius} R_j \\
\Longrightarrow & A \ge \big(\operatorname{radius} R_j\big)^{-1} \\
\Longrightarrow & A \ge \limsup_{j \to \infty} R_j^{1/j} \\
\Longrightarrow & A \ge \inf_{J \ge 0} \sup_{j \ge J} R_j^{1/j} \\
\Longrightarrow & \forall \lambda > A: \exists J: \sup_{j \ge J} R_j^{1/j} < \lambda \\
\Longrightarrow & \forall \lambda > A: \exists J: \forall j \ge J: R_j^{1/j} < \lambda \\
\Longrightarrow & \forall \lambda > A: \exists J: \forall j \ge J: R_j < \lambda^j
\end{align*}
\end{verify}

Since $\hardpart^\alpha$ has a separable kernel of the form described in Section~\ref{setting:basic}, we can apply Proposition~\ref{prop:alt-general_volterra} as long as $\hardpart^\alpha$ and $\softpart^\alpha$ satisfy the required conditions from Section~\ref{setting}.
\begin{itemize}
\item Condition~\eqref{cond:sing} {\em on $\hardpart^\alpha$ is satisfied.}

Since $P$ is a monic polynomial whose roots are all simple, $1/P(-\zeta)$ has the nice partial fraction decomposition
\[ \frac{1}{P(-\zeta)} = \sum_{-\beta \in \roots} \frac{1}{-P'(-\beta)\,(\zeta - \beta)}, \]
where $\roots$ is the zero set of $P$ \cite[Section 1.4, Exercise~2]{ahlfors}. 
It follows that
\begin{align*}
-\frac{Q\big(-\zeta(a')\big)}{P\big(-\zeta(a)\big)} & = \sum_{-\beta \in \roots} \frac{Q\big(-\zeta(a')\big)}{P'(-\beta)\,\big(\zeta(a) - \beta\big)} \\
& = \sum_{\substack{-\beta \in \roots \\ \beta \neq \alpha}} \frac{Q\big(-\zeta(a')\big)}{P'(-\beta)\,(\zeta(a) - \beta)} + \frac{Q\big(-\zeta(a')\big)}{P'(-\alpha)\,(\zeta(a) - \alpha)}.
\end{align*}
Now we expand $Q$ around $-\alpha$, using the decomposition
\[ Q(-\zeta) = Q(-\alpha)+(\zeta-\alpha)\,Q_\alpha(-\zeta) \]
to split up the last term of the expression above:
\begin{multline}\label{eqn:split-numerator}
-\frac{Q\big(-\zeta(a')\big)}{P\big(-\zeta(a)\big)} = \\
\sum_{\substack{-\beta \in \roots \\ \beta \neq \alpha}} \frac{Q\big(-\zeta(a')\big)}{P'(-\beta)\,\big(\zeta(a) - \beta\big)} + \frac{\big(\zeta(a') - \alpha\big)\,Q_\alpha\big(-\zeta(a')\big)}{P'(-\alpha)\,\big(\zeta(a) - \alpha\big)} + \frac{Q(-\alpha)}{P'(-\alpha)\,\big(\zeta(a) - \alpha\big)}.    
\end{multline}
First, we want to show that $-q/p-\tau_\alpha/(\zeta-\alpha)$ is bounded on a neighborhood of $\zeta = \alpha$ in $\domain_\alpha$. It follows from equation~\eqref{eqn:split-numerator} that
\begin{align*}
-\frac{Q(-\zeta)}{P(-\zeta)} & =
\sum_{\substack{-\beta \in \roots \\ \beta \neq \alpha}} \frac{Q(-\zeta)}{P'(-\beta)\,(\zeta - \beta)} + \frac{(\zeta - \alpha)\,Q_\alpha(-\zeta)}{P'(-\alpha)\,(\zeta - \alpha)} + \frac{Q(-\alpha)}{P'(-\alpha)\,(\zeta - \alpha)} \\
&=\sum_{\substack{-\beta \in \roots \\ \beta \neq \alpha}} \frac{Q(-\zeta)}{P'(-\beta)\,(\zeta - \beta)} + \frac{Q_\alpha(-\zeta)}{P'(-\alpha)} + \frac{Q(-\alpha)}{P'(-\alpha)\,(\zeta - \alpha)}\,.
\end{align*}
Recalling that $\tau_\alpha = Q(-\alpha)/P'(-\alpha)$, we conclude that
\[-\frac{Q(-\zeta)}{P(-\zeta)}-\frac{\tau_\alpha}{\zeta-\alpha}=\sum_{\substack{-\beta \in \roots \\ \beta \neq \alpha}} \frac{Q(-\zeta)}{P'(-\beta)\,(\zeta - \beta)} + \frac{Q_\alpha(-\zeta)}{P'(-\alpha)}\,. \]
Since $Q$ and $Q_\alpha$ are polynomials, the right-hand side is bounded on any neighborhood of $\zeta = \alpha$ that avoids the other roots of $p$ and avoids infinity.

Next, we want to show that for any $\sigma > \tau_\alpha$, the bound
\[ |\zeta(a) - \alpha| \left|\frac{q(a')}{p(a)}\right| < \sigma \]
holds when $a$ and $a'$ are close enough to $\zeta = \alpha$. Taking the absolute value of both sides of equation~\eqref{eqn:split-numerator}, and using the assumption that $\tau_\alpha$ is real and positive, we see that
\begin{align*}
\left|\frac{Q\big(-\zeta(a')\big)}{P\big(-\zeta(a)\big)}\right| & \le \\ & \hspace{-15mm} \sum_{\substack{-\beta \in \roots \\ \beta \neq \alpha}} \left|\frac{Q\big(-\zeta(a')\big)}{P'(-\beta)\,\big(\zeta(a) - \beta\big)}\right| + \left|\frac{(\zeta(a') - \alpha)\,Q_\alpha\big(-\zeta(a')\big)}{P'(-\alpha)\,\big(\zeta(a) - \alpha\big)}\right| + \left|\frac{\tau_\alpha}{\zeta(a) - \alpha}\right| \\
|\zeta(a) - \alpha|\left|\frac{Q\big(-\zeta(a')\big)}{P\big(-\zeta(a)\big)}\right| & \le \\
& \hspace{-15mm} \sum_{\substack{-\beta \in \roots \\ \beta \neq \alpha}} \left|\frac{Q\big(-\zeta(a')\big)\,\big(\zeta(a) - \alpha\big)}{P'(-\beta)\,\big(\zeta(a) - \beta\big)}\right| + \left|\frac{\big(\zeta(a') - \alpha\big)\,Q_\alpha\big(-\zeta(a')\big)}{P'(-\alpha)}\right| + \tau_\alpha\,.
\end{align*}
By bringing $\zeta(a')$ closer to $\alpha$, we can make the middle term
\[\left|\frac{\big(\zeta(a') - \alpha\big)\,Q_\alpha\big(-\zeta(a')\big)}{P'(-\alpha)}\right|\]
arbitrarily small. Similarly, by bringing $\zeta(a)$ closer to $\alpha$, we can make the sum
\[\sum_{\substack{-\beta \in \roots \\ \beta \neq \alpha}} \left|\frac{Q\big(-\zeta(a')\big)\,\big(\zeta(a) - \alpha\big)}{P'(-\beta)\,\big(\zeta(a) - \beta\big)}\right|\]
arbitrarily small. Hence, for any $\sigma > \tau_\alpha$, the bound
\[|\zeta(a)-\alpha| \, \left|\frac{Q\big(-\zeta(a')\big)}{P\big(-\zeta(a)\big)}\right| \le \sigma \]
holds over some neighborhood of $\big(\zeta(a), \zeta(a')\big) = (\alpha, \alpha)$ in $\domain_\alpha^2$.
\begin{verify}
\par Solution of Ahlfors exercise. First, observe that
\begin{align*}
P(t) &= (t + \beta_1) \cdots (t + \beta_d) \\
P'(t) &= \sum_k \left[ \prod_{j \neq k} (t + \beta_j) \right] \\
P'(-\beta_1) &= \prod_{j\neq 1} (-\beta_1 + \beta_j) 
\end{align*}
Then we see that
\begin{align*}
\frac{1}{P(t)} & = \sum_{j} \frac{\pi_j}{t + \beta_j}\\
-\frac{P'(t)}{P(t)^2} &= -\sum_{j} \frac{\pi_\beta}{(t + \beta_j)^2}\\
P'(t) &= \sum_{j} \pi_j \left(\frac{P(t)}{t + \beta_j}\right)^2\\
P'(-\beta_1) &=  \pi_1 \prod_{j\neq 1} (-\beta_1+\beta_j)^2\\
\frac{1}{P'(-\beta_1)} &= \pi_1
\end{align*}
\end{verify}
\item Condition~\eqref{cond:diag-basic} {\em on $\hardpart^\alpha$ is satisfied.}

Since we want to bound $\big(\zeta(a) - \alpha\big)\,\hardker^\alpha(a, a')$ with a function of the difference $\omega(a, a') := \zeta(a) - \zeta(a')$, we will rewrite it as a rational function of $\omega(a, a')$ and $\zeta(a)$. First, rewrite
\begin{align*}
q(a') & = Q\big(-\zeta(a')\big) \\
& = Q\big(\omega(a, a') - \zeta(a)\big)
\end{align*}
in the form
\[ q(a') = Q_{d-1}\big(\zeta(a)\big)\;\omega(a, a')^{d-1} + \ldots + Q_1\big(\zeta(a)\big)\;\omega(a, a') + Q_0\big(\zeta(a)\big), \]
where $Q_0, \ldots Q_{d-1}$ are polynomials of degree at most $d-1$. Next, knowing that $\alpha$ is a root of $p$, rewrite $p$ in the form $(\zeta - \alpha)\,P_\alpha(\zeta)$, where $P_\alpha$ is a polynomial of degree $d-1$. We can then write $-\big(\zeta(a) - \alpha\big)\,\hardker^\alpha(a, a')$ in the form
\[ \big(\zeta(a)-\alpha\big)\,\frac{q(a')}{p(a)} = \frac{Q_{d-1}\big(\zeta(a)\big)}{P_\alpha\big(\zeta(a)\big)}\;\omega(a, a')^{d-1} + \ldots + \frac{Q_1\big(\zeta(a)\big)}{P_\alpha\big(\zeta(a)\big)}\;\omega(a, a') + \frac{Q_0\big(\zeta(a)\big)}{P_\alpha\big(\zeta(a)\big)}, \]
viewing it as a polynomial in $\omega(a, a')$ whose coefficients are rational functions in $\zeta(a)$. If we keep $a$ away from the roots of $p$ other than $\zeta = \alpha$, each coefficient is bounded, so $|\zeta(a) - \alpha|\,|\hardker^\alpha(a, a')|$ is bounded by a polynomial in $|\omega(a, a')|$. It follows that for any $\lambda_\Delta > 0$, and any domain $\domain_\alpha$ that avoids all the roots of $p$ other than $\zeta = \alpha$, we have
\[ |\zeta(a)-\alpha|\,|\hardker^\alpha(a, a')| \lesssim e^{\lambda_\Delta|\omega(a, a')|} \]
over all $a, a' \in \domain_\alpha$.

For consistency with Condition~\eqref{cond:eps-lambda}, we choose $\lambda_\Delta > A$.
\item Condition~\eqref{cond:reg-p} {\em on $\hardpart^\alpha$ is satisfied.}

This is true because $p$ is a polynomial.
\item Condition~\eqref{cond:eps-lambda} {\em on $\softpart^\alpha$ is satisfied.} 

First, observe that
\begin{align*}
|\softker^\alpha(a, a')| & = \left|\frac{k_R(a, a')}{p(a)}\right| \\
& \le \frac{1}{|p(a)|} \sum_{j=0}^\infty \frac{|R_{j}|}{(j+1)!} \, |\zeta(a)-\zeta(a')|^{j+1} \\
& = \frac{|\zeta(a)-\zeta(a')|}{|p(a)|} \sum_{j=0}^\infty \frac{|R_{j}|}{(j+1)!} \, |\zeta(a)-\zeta(a')|^j.
\end{align*}
Now, choose $\lambda_\Delta > A$. Since $R(z^{-1})$ is holomorphic on the disk $|z| > A$, we have $|R_j| \lesssim \lambda_\Delta^j$ over all and $j \in \{0, 1, 2, \ldots\}$, as mentioned earlier. It follows that
\begin{align*}
\sum_{j=0}^\infty \frac{|R_{j}|}{(j+1)!} \, |\zeta(a)-\zeta(a')|^j
& \lesssim \sum_{j=0}^\infty \left(\frac{1}{j+1}\right)\frac{\lambda_\Delta^j}{j!} \, |\zeta(a)-\zeta(a')|^j \\
& \lesssim e^{\lambda_\Delta|\zeta(a)-\zeta(a')|}
\end{align*}
over all $a, a' \in \domain_\alpha^2$, so
\[ |\softker^\alpha(a, a')| \lesssim \frac{|\zeta(a)-\zeta(a')|}{|p(a)|} e^{\lambda_\Delta|\zeta(a)-\zeta(a')|} \]
over all $a, a' \in \domain_\alpha^2$. Like before, rewrite $p$ in the form $(\zeta - \alpha)\,P_\alpha(\zeta)$, so we have
\[ |\softker^\alpha(a, a')| \lesssim \frac{|\zeta(a)-\zeta(a')|}{|\zeta(a) - \alpha|\,\big|P_\alpha\big(\zeta(a)\big)\big|} e^{\lambda_\Delta|\zeta(a)-\zeta(a')|} \]
over all $a, a' \in \domain_\alpha$. Since $P_\alpha$ is a polynomial with a finite number of roots, and $\domain_\alpha$ does not touch any of the roots, $|P_\alpha(\zeta)|^{-1}$ is bounded on $\domain_\alpha$. We conclude that
\[ |\softker^\alpha(a, a')| \lesssim \frac{|\zeta(a)-\zeta(a')|}{|\zeta(a) - \alpha|} e^{\lambda_\Delta|\zeta(a)-\zeta(a')|} \]
over all $a, a' \in \domain_\alpha$, so Condition~\eqref{cond:eps-lambda} is satisfied with $\gamma = 1$.
\end{itemize}
We can now apply Proposition~\ref{prop:alt-general_volterra}, yielding the desired result.
\end{proof}

\bibliographystyle{siamplain}
\bibliography{reg-sing-volterra.bib}
\end{document}